\documentclass[a4paper,11pt]{amsart}
\usepackage{latexsym,amscd,amssymb,url}
\usepackage[all,cmtip]{xy}  
\usepackage{graphicx}
\usepackage{xcolor}
\usepackage{enumitem}
\definecolor{dblue}{rgb}{0,0,.6}
\usepackage[colorlinks=true, linkcolor=dblue, citecolor=dblue, filecolor = dblue, menucolor = dblue, urlcolor = dblue]{hyperref}

\newtheorem{theorem}{Theorem}[section]

\theoremstyle{plain}

\newtheorem{corollary}[theorem]{Corollary}

\newtheorem{lemma}[theorem]{Lemma}

\newtheorem{proposition}[theorem]{Proposition}
\newtheorem{remark}[theorem]{Remark}

\numberwithin{equation}{section}

\newcommand{\del}{\partial}
\newcommand{\Z}{\mathbb Z}
\newcommand{\Q}{\mathbb Q}

\newcommand{\C}{\mathbb C}

\newcommand{\CP}{\mathbb P}

\newcommand{\im}{\operatorname{im}}

\newcommand{\Pic}{\operatorname{Pic}}

\newcommand{\Spec}{\operatorname{Spec}}

\newcommand{\Gal}{\operatorname{Gal}}

\newcommand{\codim}{\operatorname{codim}}

\newcommand{\Br}{\operatorname{Br}}

\newcommand{\CH}{\operatorname{CH}}

\newcommand{\cl}{\operatorname{cl}} 
 
\newcommand{\Frac}{\operatorname{Frac}}

  \newcommand{\coker}{\operatorname{coker}}
  \newcommand{\Griff}{\operatorname{Griff}}  
\newcommand{\tors}{\operatorname{tors}}

\newcommand{\alg}{\operatorname{alg}}

\newcommand{\dashedlongrightarrow}{\xymatrix@1@=15pt{\ar@{-->}[r]&}}
\renewcommand{\longrightarrow}{\xymatrix@1@=15pt{\ar[r]&}}
\renewcommand{\mapsto}{\xymatrix@1@=15pt{\ar@{|->}[r]&}}
\renewcommand{\twoheadrightarrow}{\xymatrix@1@=15pt{\ar@{->>}[r]&}}
\newcommand{\hooklongrightarrow}{\xymatrix@1@=15pt{\ar@{^(->}[r]&}}
\newcommand{\congpf}{\xymatrix@1@=15pt{\ar[r]^-\sim&}}
\renewcommand{\cong}{\simeq}

\begin{document}    

\title[Infinite torsion in Griffiths groups]{Infinite torsion in Griffiths groups}

\author{Stefan Schreieder} 
\address{Institute of Algebraic Geometry, Leibniz University Hannover, Welfengarten 1, 30167 Hannover , Germany.}
\email{schreieder@math.uni-hannover.de}

\date{\today}
\date{January 28, 2022;}
%


\begin{abstract}   
We show that there are smooth complex projective varieties with infinite $2$-torsion in their third Griffiths groups. 
It follows that the torsion subgroup of  Griffiths groups is in general not finitely generated, thereby solving a problem of Schoen from 1992.
\end{abstract}

\maketitle 
 
\section{Introduction}

The Griffiths group $\Griff^i(X)$ of a smooth complex projective variety $X$   is the group of homologically trivial codimension $i$  cycles modulo algebraic equivalence.  
This is a countable abelian group which is a basic invariant of $X$.  
However, detecting whether a given homologically trivial cycle is nontrivial in the Griffiths group is a subtle problem.  
For instance, the isomorphism type of the abelian group $\Griff^i(X)$ is not known in any nontrivial example.

Griffiths \cite{griffiths} used his transcendental Abel--Jacobi maps to construct the first example of a smooth complex projective variety with nontrivial Griffiths group.
Clemens \cite{clemens} combined Griffiths' approach with a degeneration argument to show that in fact  $\Griff^i(X)\otimes \Q$ may be an infinite dimensional $\Q$-vector space for any $i\geq 2$.
This showed that Griffiths groups are in general not finitely generated modulo torsion.

Improving earlier results of Schoen \cite{schoen-modn} and Rosenschon--Srinivas  \cite{RS}, Totaro \cite{totaro-chow} showed that $\Griff^i(X)/\ell$ may for $i\geq 2$ be infinite for any prime $\ell$.  
These results rely on Griffiths' method, a theorem of Bloch--Esnault \cite{BE}, and ideas from Nori's proof \cite{nori-2} of Clemens' theorem.  

Schoen \cite{schoen-torsion} used Griffiths' method to show that Griffiths groups may contain nontrivial torsion. 
The first nontrivial torsion classes with trivial transcendental  Abel--Jacobi invariants have been constructed by Totaro \cite{totaro-IHC} via a topological method; non-torsion classes with that property had earlier been constructed by Nori \cite{nori}.  

\subsection{Main result}

This article shows that the theory of refined unramified cohomology developed in \cite{Sch-torsion1} furnishes a new method to detect nontriviality of classes in the Griffiths group.
As a concrete application, we prove the following.

\begin{theorem}\label{thm:JC}
Let $JC$ be the Jacobian of a very general quartic curve $C\subset \CP^2_\C$.  
Then for any very general Enriques surface $X$ over $\C$, $\Griff ^3(X\times JC)$ has infinite $2$-torsion. 
\end{theorem}

As an immediate corollary, we obtain:

\begin{corollary} \label{cor:fg}
The torsion subgroup of  Griffiths groups of smooth complex projective varieties is in general not finitely generated. 
\end{corollary}

The above results solve a problem of Schoen \cite{schoen-torsion}, who writes in the introduction: 

``\textit{Although our experience is that torsion in the Griffiths groups of varieties over $\C$ is difficult to find, we have no compelling evidence that it is always finite or usually zero.}''

Schoen's problem and Theorem \ref{thm:JC} above should be compared to a theorem of Merkurjev--Suslin \cite{MS}, who showed that for any integer $n$, the $n$-torsion subgroup of $\Griff^2$ is finite.

The analogue of Corollary \ref{cor:fg} is trivial for Chow groups, as any elliptic curve $E$ has infinite torsion in $\CH_0(E)$.
However, even for Chow groups, the problem becomes interesting if we restrict to $\ell$-torsion for a given prime $\ell$.
It has been shown in \cite{schoen-product,RS,totaro-chow} that these groups may be infinite, but  
 all involved torsion cycles are algebraically equivalent to zero.

The infinitely many different $2$-torsion classes in $\Griff^3(X\times JC)$ from Theorem \ref{thm:JC} are given by exterior products $K_X\times z$, where $K_X\in \Pic X$ is the unique $2$-torsion class  on the Enriques surface $X$ and where $z$ is the Ceresa cycle $C-C^{-}$ on $JC$ (see \cite{ceresa}), or a pullback of the Ceresa cycle by one of infinitely many isogenies.
We will show that infinitely many of the cycles $K_X\times z$ are linearly independent modulo $2$.
To this end we will use Totaro's result \cite{totaro-chow}, who showed that infinitely many of the classes $z$ are linearly independent modulo $2$ in $\Griff^2(JC)$.
Totaro used among other ingredients that the Ceresa cycle has nontrivial Abel--Jacobi invariant, cf.\ \cite{hain}.
In contrast, $K_X\times z$ has trivial Abel--Jacobi invariant.

While it is natural to consider cycles of the form $K_X\times z$ as above, our method of proving that they are nontrivial (resp.\ linearly independent)  
is new.
The main idea is that  
  \cite{Sch-torsion1} allows one to use cohomological tools (most notably Poincar\'e duality)  that are a priori not available within the framework of algebraic cycles. 

Even though Theorem \ref{thm:JC} concentrates on $2$-torsion in the Griffiths group, our method is flexible and large parts of this paper work for $\ell$-torsion for arbitrary primes $\ell$.
In particular, it is conceivable that our method will allow to find infinite $\ell$-torsion for other primes $\ell$ by replacing Enriques surfaces by suitable surfaces $X$ with $\ell$-torsion in $H^2(X,\Z)$.

\subsection{An injectivity theorem}
Let $Y$ be an arbitrary smooth complex projective variety and let $A^i(Y):=\CH^i(Y)/\sim_{\alg}$ be the Chow group modulo algebraic equivalence.
Theorem \ref{thm:JC} will be deduced from the following, which is the main result of this paper.

\begin{theorem} \label{thm:main}
Let $Y$ be a smooth complex projective variety and let $X$ be an Enriques surface that is very general with respect to $Y$.
Then the exterior product map
$$
A^{i}(Y)/2 \longrightarrow A^{i+1}(X\times Y)[2] ,\ \ [z]\mapsto [K_X\times z]
$$
is injective.  In fact,  $[K_X\times z]$ is  not divisible by  $2$ in $A^{i+1}(X\times Y)$ unless  $[z]=0\in A^{i}(Y)/2$.
\end{theorem}

The condition on $X$ means that it lies outside a countable union of proper closed subsets (which may depend on $X$) of the moduli space of Enriques surfaces.

We let $\mathcal T^{i+1}(X\times Y)\subset \Griff^{i+1}(X\times Y)_{\tors}$ denote the kernel of the transcendental Abel--Jacobi map (see (\ref{def:lambda_tr}) below) and define
$$
E_2^i(Y):=\ker(\cl_Y^i:A^i(Y)/2\longrightarrow H^{2i}(Y,\Z/2)).
$$ 

\begin{corollary} \label{cor:main}
In the notation of Theorem \ref{thm:main},  there is a canonical injection
$$
E_2^{i}(Y) \hookrightarrow \mathcal T^{i+1}(X\times Y)[2] ,\ \ [z]\mapsto [K_X\times z] .
$$
\end{corollary}

By \cite{Sch-torsion1} (see Theorem \ref{thm:refined} below),  there is  a canonical extension 
\begin{align} \label{def:E}
0\longrightarrow \Griff^i(Y)/2 \longrightarrow E_{2}^i(Y)\longrightarrow 
 Z^{ i}(Y)[2]/H^{2i} (Y,\Z )[2] 
\longrightarrow 0 ,
\end{align} 
where $Z^{ i}(Y)[2]$ denotes the $2$-torsion subgroup of $\coker(\cl_Y^i:\CH^i(Y)\to H ^{{2i}}(Y,\Z ))$. 
Corollary \ref{cor:main} thus shows that there are two sources for nontrivial $2$-torsion classes with trivial transcendental Abel--Jacobi invariants  in the Griffiths group of $X\times Y$: one coming from $\Griff^i(Y)/2$ and one stemming from  $Z^{ i}(Y)[2]/H^{2i} (Y,\Z )[2] $.
While the former may be infinite \cite{totaro-chow,diaz-JAG}, the latter is always a finite group.
One source for nontrivial elements in $Z^{ i}(Y)[2]/H^{2i} (Y,\Z )[2] $ are non-algebraic non-torsion Hodge classes $\alpha$ on $Y$ such that $2\alpha$ is algebraic. 
Several examples with that property are known (see e.g.\ \cite{OS,diaz}) and we will discuss those applications in Corollaries \ref{cor:CH-0=0-2},  \ref{cor:CH-0=0}, and \ref{cor:Kummer} below.
For instance, Corollary \ref{cor:main} and \cite{OS} lead to the first known example of a smooth complex projective variety $X$ with a rational decomposition of the diagonal for which $\mathcal T^3(X)\neq 0$,  see Corollaries \ref{cor:CH-0=0-2} and \ref{cor:CH-0=0} below.

In \cite{SV}, Soul\'e and Voisin constructed non-divisible torsion classes in the Griffiths group of products $X\times Y$, where $H^2(X,\Z)_{\tors}\neq 0$ and $Y$ is a carefully chosen hypersurface in $\CP^4$ so that the integral Hodge conjecture fails for $Y$ by Koll\'ar's argument \cite{BCC}. 
Their argument relies  on degenerations of $Y$. 
Instead, we will degenerate $X$,  while $Y$ may be arbitrary and the contribution of $\Griff^i(Y)/2$ will taken into account in a precise way.  

\subsection{Degenerations of Enriques surfaces}
In addition to the theory of refined unramified cohomology, Theorem \ref{thm:main} relies on the following geometric input.

\begin{theorem}\label{thm:Enriques-deg-intro}
There is a regular flat projective scheme $\mathcal X\to \Spec R$ over a discrete valuation ring $R$ whose residue field is an algebraically closed field $\kappa$ of characteristic zero, such that:
\begin{enumerate}
\item the geometric generic fibre $X_{\overline \eta}$ is an Enriques surface;\label{item:Enriques:1}
\item the special fibre $X_0=\mathcal X\times \kappa$ is a union of ruled surfaces;\label{item:Enriques:ruled}
\item the restriction map 
$\Br(\mathcal X)[2]\longrightarrow \Br(X_{\overline \eta})[2]$ 
 is surjective.
\label{item:Enriques:surjective}
\end{enumerate}
\end{theorem}

The geometric meaning of the theorem is as follows.
By \cite{dJ}, the unique nonzero class in $ \Br(X_{\overline \eta})\cong \Z/2$ corresponds to a smooth (i.e.\ unramified) conic bundle $P\to X_{\overline \eta}$ and the above theorem implies that this conic bundle extends to a smooth  conic bundle $\mathcal P\to \mathcal X$.
That is, while the Enriques surface breaks up into ruled components, the conics in the fibration $P\to X_{\overline \eta}$ remain smooth and do not break up into the union of two lines. 

The fact that degenerations as above exist was a surprise to the author.  
To explain one subtle aspect,  note that we may   assume that the dvr $R$ in Theorem \ref{thm:Enriques-deg-intro}  is complete.
Item (\ref{item:Enriques:1}) together with the proper base change theorem can then be used to show $\Br(\mathcal X)_{\tors} \cong \Br(X_0)_{\tors}$ (see Proposition \ref{prop:Br-neq-0} and Remark \ref{rem:prop:Brauer} below) and so $ \Br(X_0)[2]\neq 0$  by item (\ref{item:Enriques:surjective}).
On the other hand, each component $X_{0i}$ of $X_0$ is ruled by item (\ref{item:Enriques:ruled}) and so $\Br(X_{0i})=0$ for all $i$.

For us, the crucial consequence of the above theorem will be as follows: 
 
 \begin{corollary}\label{cor:Enriques-deg-intro}
In the notation of Theorem \ref{thm:Enriques-deg-intro}, there is a class $\alpha\in \Br(\mathcal X)[2]$ such that $\alpha|_{X_{\overline \eta}}$ generates $ \Br(X_{\overline \eta})\cong \Z/2$ and for any component $X_{0i}$ of the special fibre: 
$
\alpha|_{X_{0i}}=0\in \Br(X_{0i}) .
$
\end{corollary}

\subsection{Outline of the argument}
Let $X$ and $Y$ be as in Theorem \ref{thm:main}.
The cohomological analogue of the exterior product map in Theorem \ref{thm:main} is the cup product map 
$$
H^2(X,\Z/2)\otimes H^{2i}(Y,\Z/2)\longrightarrow H^{2i+2}(X\times Y,\Z/2),\ \ \ \alpha\otimes \beta \mapsto p^\ast \alpha\cup q^\ast \beta,
$$
where $p:X\times Y\to X$ and $q:X\times Y\to Y$ denote the natural projections.
If $\alpha\in H^2(X,\Z/2)$ is nonzero,   Poincar\'e duality yields a class $\bar \alpha\in H^{2}(X,\Z/2)$ with $\alpha\cup \bar \alpha=\cl_X^{2}(pt)$ and we find
$$
q_\ast (p^\ast \bar\alpha\cup p^\ast \alpha\cup q^\ast \beta)=q_\ast( p^\ast \cl_X^{2}(pt)\cup q^\ast \beta)=\beta \in H^{2i}(Y,\Z/2).
$$
This classical argument shows that $p^\ast \alpha\cup q^\ast \beta\neq 0$ as long as $\alpha$ and $\beta$ are nonzero.

We aim to use this approach to prove Theorem \ref{thm:main}.
The obvious obstacle in doing so is that $A^1(X)/2$ does not satisfy Poincar\'e duality (in fact, $K_X\cdot D=0$ for any divisor $D$ on $X$).
The crucial input which allows us to circumvent this problem is  \cite{Sch-torsion1}, which shows that Chow groups modulo algebraic equivalence can  be computed by refined unramified cohomology.
An important observation here is that the Poincar\'e dual $\bar \alpha \in H^2(X,\Z/2)$ of the class $\alpha:=[K_X]\in H^2(X,\Z/2)$ generates the second unramified cohomology of $X$:
$$
\Br(X) \cong H^2_{nr}(X,\Z/2) = \Z/2 [\bar \alpha] .
$$
Hence, algebraic cycles modulo algebraic equivalence as well as the Poincar\'e dual of $K_X$ live in the world of refined unramified cohomology and so it is natural to try to work there.

Following this approach, it is straightforward to prove injectivity of the exterior product map in Theorem \ref{thm:main} on the level of cycles (not modulo algebraic equivalence) purely in terms of cohomology. 
While that statement is of course trivial,  passing to algebraic equivalence will introduce an error term of the form
\begin{align}\label{eq:error}
q_\ast(p^\ast \bar \alpha\cup \iota_\ast \xi )\in H^{2i-1}(V,\Z/2),\ \ \ \text{where}\ \ \ \xi\in   \bigoplus_{w\in (X\times Y)^{(i)}} H^1( w ,\Z/2). 
\end{align}
Here $V\subset Y$ is some open subset whose complement $R=Y\setminus V$ has codimension at least $i-1$.
The main technical difficulty in the proof of Theorem \ref{thm:main} is to show that (\ref{eq:error}) vanishes (possibly up to a class that extends to all of $Y$).

At this point we use specialization maps for refined unramified cohomology that we construct in Section \ref{sec:sp} below and which serve as a replacement for Fulton's specialization maps on Chow groups \cite{fulton}.
These maps are inspired by well-known specialization maps in Galois cohomology; see also Remark \ref{rem:sp} below. 

Even though we would like to show that (\ref{eq:error}) vanishes for $X$ and $Y$ smooth, we will be able to reduce the problem to the situation where $X$ splits up into many components as in Theorem \ref{thm:Enriques-deg-intro}. 
 This may be surprising, as usually one cannot prove the vanishing of an invariant by showing that it vanishes after specialization, but we will be able to put ourselves in a situation where that argument actually works. 

At this final step, the geometry of Enriques surfaces via Theorem \ref{thm:Enriques-deg-intro} comes in, as it is exactly the kind of degeneration needed to ensure that $\bar \alpha$ vanishes on each component of the special fibre of the degeneration and we will show in Lemma \ref{lem:vanishing} that this implies that the specialization of (\ref{eq:error}) vanishes.
This concludes the proof of Theorem \ref{thm:main} up to the proof of Theorem \ref{thm:Enriques-deg-intro}.
The latter relies in turn on an analysis of flower pot degenerations of Enriques surfaces, constructed by Persson and Horikawa  \cite{persson,horikawa,morrison}, cf.\ Theorem \ref{thm:Enriques-flower-pot} below.

\section{Preliminaries}

\subsection{Conventions}
For an abelian group $G$, we denote by $G[\ell ^r]$ the subgroup of $\ell^r$-torsion elements.
 Whenever $G$ and $H$ are abelian groups so that there is a canonical map $H\to G$ (and there is no reason to confuse this map with a different map), we write $G/H$ as a short hand for $\coker(H\to G)$.

All schemes are separated.
An algebraic scheme $X$ is a scheme of finite type over a field.
Its Chow group of codimension $i$ cycles modulo rational equivalence is denoted by $\CH^i(X)$; the quotient of $\CH^i(X)$ modulo algebraic equivalence is denoted by $A^i(X)$.

A variety  is an integral scheme of finite type over a field. 
For an equi-dimensional algebraic scheme $X$, we denote by $X^{(i)}$ the set of all codimension $i$ points of $X$.
For a scheme $X$ over a field $k$, we write for any field extension $K$ of $k$ the scheme given by extension of scalars as $X_K:=X\times _kK$.
A very general point of a scheme over $\C$ is a closed point outside a countable union of proper closed subsets.

If $R$ is an integral local ring with residue field $\kappa$ and fraction field $K$, with algebraic closure $\overline K$, then for any flat $R$-scheme $\mathcal X\to \Spec R$, we write $X_0:=\mathcal X\times_R\kappa$ (resp.\ $X_{\overline 0}:=\mathcal X\times_R \overline \kappa$) for the special (resp.\ geometric special) fibre and $X_{\eta}:=\mathcal X\times_RK$ (resp.\ $X_{\overline \eta}:=\mathcal X\times_R \overline K$) for the generic (resp.\ geometric generic)  fibre.

An irreducible flat  $R$-scheme $\mathcal X\to \Spec R$ over a discrete valuation ring $R$ is called strictly semi-stable, if $\mathcal X$ is regular, the generic fibre $X_\eta$ is smooth and the special fibre $X_0$ is a geometrically reduced simple normal crossing divisor on $\mathcal X$, i.e.\ the components of $X_0$ are smooth and the scheme-theoretic intersection of $r$ different components of $X_0$ is either empty or smooth and equi-dimensional of codimension $r$ in $\mathcal X$.

\subsection{Cohomology}
For a scheme $X$ and a prime $\ell$ invertible on $X$ we write
\begin{align} \label{def:Hi}
H^i(X,\mu_{\ell^r}^{\otimes n}):=H^i(X_{\text{\'et}},\mu_{\ell^r}^{\otimes n}) \ \ \text{and} \ \ H^i(X,\Z_{\ell}(n)):=H^i_{cont}(X_{\text{\'et}},\Z_{\ell}(n)) ,
\end{align}
where $H^i_{cont}$ denotes Jannsen's continuous \'etale cohomology, see \cite{jannsen}; sometimes we also write $H^i(X,\Z/\ell^r(n)) $ in place of $H^i(X,\mu_{\ell^r}^{\otimes n})$.
These groups are functorial with respect to pullbacks along arbitrary morphisms.

For a ring  $A$, we write $H^i(A,\mu_{\ell^r}^{\otimes n}):=H^i(\Spec A,\mu_{\ell^r}^{\otimes n})$.
If $A=K$ is a field, then these groups coincide with the Galois cohomology of the absolute Galois group of $K$.

For a scheme $X$ and a point $x\in X$, we write
$$
H^i(x,\mu_{\ell^r}^{\otimes n}):= \lim_{\substack{\longrightarrow \\ \emptyset \neq V_x\subset \overline {\{x\}}}} H^i(V_x,\mu_{\ell^r}^{\otimes n}),
$$
where $V_X$ runs through all open dense subsets of the closure of $x$.
By  \cite[p.  88, III.1.16]{milne}, this direct limit coincides with the \'etale cohomology of $\Spec \kappa(x)$ and hence with the Galois cohomology of the field $\kappa(x)$.

\begin{remark}\label{rem:Borel-Moore}
If $X$ is a smooth equi-dimensional algebraic scheme over a field $k$, then the groups in (\ref{def:Hi}) agree with the Borel--Moore cohomology groups used in \cite{Sch-torsion1},  see \cite[Lemma 6.5 and Proposition 6.6]{Sch-torsion1}.
\end{remark}

\begin{lemma}\label{lem:pushforward}
Let $f:X\to Y$ be a proper morphism between equi-dimensional smooth algebraic schemes over a field $k$.
Let $\ell$ be a prime invertible in $k$ and let $c:=\dim Y-\dim X$.
Then for $A=\Z/\ell^r$ or $A=\Z_\ell$, there are pushforward maps
$$
f_\ast:H^{i-2c}(X,A(n-c))\longrightarrow H^{i}(Y,A(n))
$$
that are compatible with respect to open immersions.
\end{lemma}
\begin{proof}
This follows from Remark \ref{rem:Borel-Moore} and  \cite[(P1)  in Proposition 6.6]{Sch-torsion1}.
\end{proof}

\begin{lemma}\label{lem:Gysin}
Let $X$ be a smooth equi-dimensional algebraic scheme over a field $k$ and let $Z\subset X$ be a smooth closed subscheme of pure codimension $c$ and with complement $U$.
Then for $A=\Z/\ell^r$ or $A=\Z_\ell$,  there is a long exact Gysin sequence
\begin{align}\label{eq:Gysin}
H^i(X,A(n))\longrightarrow H^i(U,A(n))\stackrel{\del}\longrightarrow H^{i+1-2c}(Z,A(n-c))\stackrel{\iota_\ast}\longrightarrow H^{i+1}(X,A(n)) ,
\end{align}
where $\iota_\ast$ is the pushforward map from Lemma \ref{lem:pushforward}.
This sequence is functorial with respect to pullbacks along open immersions as well as pushforwards along proper morphisms  $f:X'\to X$ such that $X'$ and $Z'=f^{-1}(Z)$ are smooth and equi-dimensional.
\end{lemma}
\begin{proof}
This follows from Remark \ref{rem:Borel-Moore} and  \cite[(P2)  in Proposition 6.6]{Sch-torsion1}.
\end{proof}

The Gysin sequence is compatible with cup products in the following sense. 
If $\alpha\in H^j(X,A(m))$, then cup product with $\alpha$ induces a commutative diagram
\begin{align}\label{eq:Gysin-cup}
\xymatrix{
H^i(X,A(n))\ar[r]\ar[d]^{\cup \alpha} & H^i(U,A(n))\ar[d]^{\cup \alpha|_U}\ar[r]^-{\del} & H^{i+1-2c}(Z,A(n-c))\ar[d]^{\cup \alpha|_{Z}}\\
H^{i+j}(X,A(n+m))\ar[r]  & H^{i+j}(U,A(n+m)) \ar[r]^-{\del} & H^{i+j+1-2c}(Z,A(n+m-c)) .
}
\end{align}
This fact is well-known; the case $A=\Z/\ell^r$ (which is enough for  the purpose of this paper) is for instance spelled out in \cite[Lemma 2.4]{Sch-survey}.

\begin{remark} \label{rem:mathcal-X} 
For $A=\Z/\ell^r$,  $H^i(X,A(n))$ commutes with filtered inverse limits of schemes with affine transition maps, see \cite[p.\ 88, III.1.16]{milne}.
Using this, it follows that Lemma \ref{lem:Gysin} and (\ref{eq:Gysin}) remain true in the case where $X $ is a regular scheme of finite type  
over a local ring $R=\mathcal O_{B,0}$ of a smooth $k$-variety $B$ at a closed point $0\in B$ and $Z\subset   X$ is a smooth equi-dimensional $k$-variety contained in the special fibre of $ X\to \Spec R$. 
\end{remark}

\subsection{Classical  unramified cohomology and Brauer groups}

Let $X$ be an integral regular scheme and let $x\in X^{(1)}$ be a codimension one point such that $\ell$ is invertible in the residue field $\kappa(x)$. 
Then there is a residue map in Galois cohomology
\begin{align}\label{eq:del_x}
\del_x:H^i(k(X),\mu_{\ell^r}^{\otimes n})\longrightarrow H^{i-1}(\kappa(x),\mu_{\ell^r}^{\otimes n-1}) ,
\end{align}
where $k(X)$ denotes the residue field of the generic point of $X$.
If $\ell$ is invertible in each residue field of $X$, then the unramified cohomology of $X$ may be defined as follows, see \cite[Theorem 4.1.1(a)]{CT}:
$$
H^i_{nr}(X,\mu_{\ell^r}^{\otimes n}):=\left\lbrace \alpha\in H^i(k(X),\mu_{\ell^r}^{\otimes n})\mid \del_x\alpha=0 \ \ \forall x\in X^{(1)} \right\rbrace .
$$ 

For a scheme $X$, we let $\Br(X):=H^2(X_{\text{\'et}},\mathbb G_m)$. 
For any prime $\ell$ that is invertible on $X$, the Kummer sequence yields an isomorphism
\begin{align}\label{eq:Brauer}
\coker(c_1:\Pic(X)\longrightarrow H^2(X_{\text{\'et}},\mu_{\ell^r}))\cong \Br(X)[\ell^r].
\end{align}
If $X$ is a regular variety, then $\Br(X)[\ell^r]\cong H^2_{nr}(X,\mu_{\ell^r})$, see \cite[Proposition 4.2.3]{CT}. 

\subsection{Refined unramified cohomology}
In this paper we use refined unramified cohomology only for smooth equi-dimensional algebraic schemes.
In this case the cohomology theory in (\ref{def:Hi}) agrees with Borel--Moore cohomology used in \cite{Sch-torsion1}, see Remark \ref{rem:Borel-Moore}.
For an equi-dimensional algebraic scheme $X$, we denote by $F_\ast X$ the increasing filtration given by
$$
 F_0X\subset F_1X\subset \dots \subset F_{\dim X}X=X,
\ \ \ \text{where}\ \ \ 
F_jX:=\{x\in X\mid \codim_X(x)\leq j\}.
$$ 
Fix a prime $\ell$ invertible on $X$ and let $A=\Z/\ell^r$ or $A=\Z_\ell$.
If $X$ is smooth (or more generally, if $F_jX$ is contained in a smooth open subset of $X$), we define
$$
 H^i(F_jX,A(n)):=\lim_{\substack{\longrightarrow \\  F_jX\subset U\subset X  }} H^i(U,A(n)) ,
$$
where $U$ runs through all open subsets of $X$ that contain $F_jX$, and where the cohomology functor is (continuous) \'etale cohomology from (\ref{def:Hi}).

\begin{lemma} \label{lem:funct-FjX}
Let $f:X\to Y$ be a morphism between smooth equi-dimensional algebraic schemes.
If $f$ is flat, then there is a pullback map
$$
f^\ast:H^i(F_jY,A(n)) \longrightarrow H^i(F_{j}X,A(n)).
$$
If $f$ is proper of relative codimension $c:=\dim (Y)-\dim(X)$, there is a pushforward map 
$$
f_\ast:H^{i}(F_{j}X,A(n )) \longrightarrow H^{i+2c}(F_{j+c}Y,A(n+c)) .
$$
\end{lemma}
\begin{proof} 
Let $V\subset Y$ be an open subset with $F_{j}Y\subset V$.
The complement $W=Y\setminus V$ has codimension at least $j+1$.
If $f$ is flat, then $f^{-1}(W)\subset X$ has codimension at least $j+1$ as well.
Hence there is a well-defined pullback map
$
H^i(V,A(n))\to H^i(F_{j}X,A(n))
$.
This map is compatible with respect to restrictions to smaller open subsets $V'\subset V$ with $F_jY\subset V'$ and hence induces the pullback $f^\ast$ stated in the lemma.

Assume now that  $f$ is proper and let $U\subset X$ be an open subset with $F_{j }X\subset U$ and complement $Z=X\setminus U$.
 Then $f(Z)\subset Y$ has codimension at least $j+c+1$.
 Moreover, 
Lemma \ref{lem:pushforward} induces a pushforward map
$$
f_\ast : H^{i }(X\setminus f^{-1}(f(Z)),A(n ))\longrightarrow H^{i+2c}(Y\setminus f(Z),A(n+c)).
$$
Pre- and postcomposing this with the canonical restriction maps, we get a pushforward map
$$
f_\ast:H^{i }(U,A(n))\longrightarrow  H^{i+2c}(F_{j+c}Y,A(n+c)).
$$
By the compatibility of pushforwards with restrictions along open immersions (see Lemma \ref{lem:pushforward}), this map is compatible with respect to restriction from $U$ to a smaller open subset $U'\subset X$ with $F_jX\subset U'\subset U$.
This implies the existence of   $f_\ast$ as claimed in the lemma.
\end{proof}

We define the $j$-th refined unramified cohomology of $X$ with values in $A(n)$ by
  $$
   H^i_{j,nr}( X,A(n)):=\im ( H^i(F_{j+1}X,A(n))\to  H^i(F_jX,A(n))) .
  $$
As indicated above, these  groups coincide with the refined unramified cohomology groups from \cite{Sch-torsion1} defined via Borel--Moore cohomology,  cf.\ Remark \ref{rem:Borel-Moore}.

Taking direct limits over (\ref{eq:Gysin}) twice, we get  the following, see  \cite[Lemma 5.8]{Sch-torsion1}. 

\begin{lemma} \label{lem:les}
Let $X$ be a smooth equi-dimensional algebraic scheme  over $k$.
For any $j,n\in \Z$, there is a long exact sequence
{\small
\begin{align*}
\dots \longrightarrow   H^i(F_jX,A(n)) \longrightarrow   H^i(F_{j-1} X,A(n)) \stackrel{\del}\longrightarrow \bigoplus_{x\in X^{(j)}}H^{i+1-2j}(x,A(n-j)) \stackrel{\iota_\ast} \longrightarrow     H^{i+1}(F_jX,A(n)) \longrightarrow \dots,
\end{align*}
}
where $\iota_\ast$ (resp.\ $\del$) is induced by the pushforward (resp.\ residue) map from (\ref{eq:Gysin}).
\end{lemma}

The above lemma implies $ H^i_{0,nr}(X,\mu_{\ell^r}^{\otimes n})= H^i_{nr}(X,\mu_{\ell^r}^{\otimes n}) $, and, by 
  \cite[Corollary  5.10]{Sch-torsion1},
\begin{align} \label{eq:lem:Fj}
H^i(X,A(n)) \cong H^i(F_jX,A(n)) \ \ \text{for all $j\geq \lceil i/2\rceil $.}
\end{align}

We define a decreasing filtration $F^\ast$ on $  H^i(F_jX,A(n))$ by
$$
F^m H^i(F_jX,A(n)):= \im \left( H^i(F_mX,A(n)) \longrightarrow  H^i(F_j X,A(n))  \right) 
$$
for $m\geq j$. 
In particular,
 $ 
F^{j+1} H^i(F_jX,A(n)) = H_{j,nr}^i(X,A(n)).
 $  
 We also define (see \cite[Definition 5.4]{Sch-torsion1}) a decreasing filtration $G^\ast$ on $ H^i(F_j X,\mu_{\ell ^r}^{\otimes n})$ by
$$
\alpha\in G^m H^i(F_jX,\mu_{\ell ^r}^{\otimes n})\ \ \ \Longleftrightarrow\ \ \ \delta(\alpha)\in F^m H^{i+1}(F_jX,\Z_\ell(n)) ,
$$
where $\delta$ denotes the Bockstein map associated to $0\to \Z_\ell(n)\to \Z_\ell(n)\to \mu_{\ell^r}^{\otimes n}\to 0$.
Moreover,
$$
G^{m} H^{i}_{j,nr} ( X,\mu_{\ell^r}^{\otimes n}):=\im( G^{m} H^{i} (F_{j+1}X,\mu_{\ell^r}^{\otimes n})\to  H^{i} (F_{j}X,\mu_{\ell^r}^{\otimes n})) .
$$

Related to $\delta$, there is  also the Bockstein map
$ 
\tilde \delta 
$
that is associated to the short exact sequence  $0\to \mu_{\ell^r}^{\otimes n}\to \mu_{\ell^{2r}}^{\otimes n}\to \mu_{\ell^r}^{\otimes n}\to 0$.
The Bockstein $\tilde \delta$ is compatible with $\delta$ in the sense that $\tilde \delta(\alpha)$ is the reduction modulo $\ell^r$ of $\delta(\alpha) $.
The key feature of $\tilde \delta$  is the derivation property (see e.g.\ \cite[p.\ 304]{hatcher}), which yields
\begin{align} \label{eq:derivation-property}
\tilde \delta(\alpha\cup \beta)=\tilde \delta(\alpha)\cup \beta+(-1)^{\deg(\alpha)}\alpha\cup \tilde \delta(\beta) .
\end{align}

In analogy to $G^\ast$, we have the filtration  $\widetilde G^\ast$ 
on $ H^i(F_jX,\mu_{\ell ^r}^{\otimes n})$, given by
$$
\alpha\in \widetilde G^m H^i(F_jX,\mu_{\ell ^r}^{\otimes n})\ \ \ \Longleftrightarrow\ \ \   \tilde \delta(\alpha)\in F^m H^{i+1}(F_jX,\mu_{\ell ^r}^{\otimes n}) ,
$$
see \cite[Definition 7.24]{Sch-torsion1}.   
Moreover,
$$
\widetilde G^{m} H^{i}_{j,nr} ( X,\mu_{\ell^r}^{\otimes n}):=\im( \widetilde G^{m} H^{i} (F_{j+1}X,\mu_{\ell^r}^{\otimes n})\to  H^{i} (F_{j}X,\mu_{\ell^r}^{\otimes n})) .
$$   
The compatibility of $\delta$ and $\tilde \delta$ implies directly:
\begin{align}\label{eq:G-subsetGtilde}
G^m  H^i(F_jX,\mu_{\ell ^r}^{\otimes n}) 
\subset \widetilde G^m H^i(F_j X,\mu_{\ell ^r}^{\otimes n}) .
\end{align}

\subsection{Cycle groups in terms of refined unramified cohomology}

In \cite[\S 7]{Sch-torsion1}, several cycle groups  are computed in terms of refined unramified cohomology.
The set-up in \cite{Sch-torsion1} works for arbitrary separated schemes of finite type over a field; the corresponding cohomology functor should be taken as Borel--Moore pro-\'etale cohomology, see \cite[Section 4]{Sch-torsion1}.
We will only use refined unramified cohomology for smooth varieties over a field,  in which case we may by \cite[Lemma 6.5 and Proposition 6.6]{Sch-torsion1} take the cohomology groups from (\ref{def:Hi}).
It follows that the results of \cite[Sections 5 and 7]{Sch-torsion1} hold true in our setting.
We will repeatedly make use of this fact in what follows and in particular freely use the results from  \cite[Section 7]{Sch-torsion1}.
We also note that over algebraically closed fields,   \cite[Proposition 6.6 and Lemma 7.5]{Sch-torsion1} imply that the group $A^i(X)_{\Z_\ell}$ from \cite[Definition 7.2]{Sch-torsion1} agrees with $A^i(X)\otimes _\Z \Z_\ell$, where $A^i(X)$ denotes the Chow group of algebraic cycles modulo algebraic equivalence.

For instance, if we define 
\begin{align} \label{def:E-ell-main-2}
E_{\ell^r}^i(X):= \ker (A^i(X)/\ell ^r\longrightarrow H^{2i}(X,\mu_{\ell^r}^{\otimes i})) ,
\end{align}
where the morphism is given by the reduction modulo $\ell^r$ of the cycle class map, then \cite[Corollary 7.12 and Lemma 7.13]{Sch-torsion1} imply the following.

\begin{theorem} 
\label{thm:refined}
Let $X$ be a smooth variety over an algebraically closed field $k$ and let $\ell$ be a prime that is invertible in $k$. 
Then there is a canonical isomorphism
$$
E_{\ell^r}^i(X)\cong
H^{2i-1}_{i-2,nr}( X,\mu_{\ell^r}^{\otimes i} ) / H^{2i-1}(X,\mu_{\ell^r}^{\otimes i})  
$$
and a canonical extension
$$
0\longrightarrow \Griff^i(X)/\ell^r\longrightarrow E_{\ell^r}^i(X)\longrightarrow Z^{i}(X)[\ell^r] /H^{2i}(X,\Z_\ell(i))[\ell^r] \longrightarrow 0 ,
$$
where $Z^{i}(X)[\ell^r]:=\coker \left( \cl^i_X:\CH^i(X)_{\Z_\ell}\to H^{2i}(X,\Z_\ell(i)) \right)[\ell  ^r] $. 
\end{theorem}

\subsection{Indivisible torsion classes with trivial transcendental Abel--Jacobi invariant}

Let $X$ be a smooth  variety over an algebraically closed field $k$ and let $\ell$ be a prime invertible in $k$. 
By \cite[Section 7.5]{Sch-torsion1}, there is  a transcendental Abel--Jacobi map on torsion cycles
\begin{align}\label{def:lambda_tr}
\lambda^i_{tr}:\Griff^i(X)[\ell ^\infty]\longrightarrow H^{2i-1}(X,\Q_\ell/\Z_\ell(i))/N^{i-1}H^{2i-1}(X,\Q_\ell (i)),
\end{align}
where $N^\ast$ denotes the coniveau filtration.
If $X$ is projective, then the above map agrees with Bloch's map and if in addition $k=\C$, then it agrees with Griffiths \cite{griffiths} transcendental Abel--Jacobi map on torsion cycles, see \cite[Proposition 8.5]{Sch-torsion1}. 
We write
$$
\mathcal T^i(X)[\ell^r]:=\ker(\lambda_{tr}^i)[\ell^r] \subset \Griff^i(X)[\ell ^r].
$$

\begin{theorem} 
\label{thm:refined:Ti[ell]}
Let $X$ be a smooth variety over an algebraically closed field $k$ and let $\ell$ be a prime invertible in $k$.
Then there are subgroups $\mathcal T_0^i(X)[\ell^r] \subset \mathcal T^i(X)[\ell^r]$ with $\mathcal T^i(X)[\ell^\infty]=\bigcup_r \mathcal T^i_0(X)[\ell^r]$ and canonical isomorphisms
$$
\mathcal T_0^i(X)[\ell^r] \cong \frac{ H^{2i-2}(F_{i-2}X,\mu_{\ell^r}^{\otimes i}) }{ G^{i} H^{2i-2}(F_{i-2} X,\mu_{\ell^r}^{\otimes i})}  \cong \frac{ H^{2i-2}_{i-3,nr}( X,\mu_{\ell^r}^{\otimes i}) }{ G^{i} H^{2i-2}_{i-3,nr}(X,\mu_{\ell^r}^{\otimes i})} .
$$  
Moreover, the kernel of the canonical surjection
$$
\mathcal T_0^i(X)[\ell^r] \cong \frac{ H^{2i-2}(F_{i-2}X,\mu_{\ell^r}^{\otimes i}) }{ G^{i} H^{2i-2}(F_{i-2}X,\mu_{\ell^r}^{\otimes i})}  \twoheadrightarrow \frac{ H^{2i-2}(F_{i-2}X,\mu_{\ell^r}^{\otimes i}) }{\widetilde G^{i} H^{2i-2}(F_{i-2}X,\mu_{\ell^r}^{\otimes i})}
$$
is given by all classes in $ \mathcal T_0^i(X)[\ell^r] $ that are $\ell^r$-divisible in $A^i(X)$.
\end{theorem}  
\begin{proof}
This follows from \cite[Corollary 7.23 and Proposition 7.25]{Sch-torsion1}.
\end{proof}

\section{Product maps} 

Here and in what follows, all tensor products will be over $\Z$ if not mentioned otherwise.

\begin{lemma} \label{lem:Lambda}
Let $k$ be an algebraically closed field and let $\ell$ be a prime invertible in $k$.
Let $X$ and $Y$ be smooth varieties over $k$.
Then there is a well-defined linear map
$$
\Lambda:\tilde \delta \left(H^1(X,\mu_{\ell^r}) \right)  \otimes \frac{  H^{2i-1}_{i-2,nr}(Y, \mu_{\ell^r}^{\otimes i})}{H^{2i-1}(Y,\mu_{\ell^r}^{\otimes i}) } 
\longrightarrow \frac{ H^{2i}(F_{i-1} (X\times Y),\mu_{\ell^r}^{\otimes i+1})}{\widetilde G^{i+1}  H^{2i} (F_{i-1}(X\times Y),\mu_{\ell^r}^{\otimes i+1})}
 $$
 which on elementary tensors is given by
 $$
\Lambda([\tilde \delta(\alpha)]\otimes [\beta]):=[p^\ast \alpha \cup q^\ast \beta] 
$$ 
where $p:X\times Y\to X$ and $q:X\times Y\to Y$ denote the natural projections and where $\alpha\in H^1(X,\mu_{\ell^r})$ and $\beta\in  H^{2i-1}(F_{i-1}Y, \mu_{\ell^r}^{\otimes i})$.
\end{lemma}
\begin{proof} 
Recall $G^m \subset \widetilde G^m$ from (\ref{eq:G-subsetGtilde}).
Well-definedness in $\beta$ follows therefore from (\ref{eq:G-subsetGtilde}) together with the isomorphism
$$
\frac{ H^{2i}(F_{i-1}(X\times Y),\mu_{\ell^r}^{\otimes i+1})}{  G^{i+1}  H^{2i} (F_{i-1}(X\times Y),\mu_{\ell^r}^{\otimes i+1})}\stackrel{\cong}\longrightarrow  \frac{ H^{2i}_{i-2,nr}( X\times Y ,\mu_{\ell^r}^{\otimes i+1})}{ G^{i+1} H^{2i}_{i-2,nr} ( X\times Y ,\mu_{\ell^r}^{\otimes i+1})}
$$
 from Theorem \ref{thm:refined:Ti[ell]}. 
To prove well-definedness in $\alpha$, assume that $\tilde \delta (\alpha)=0$.
Then by the derivation property of the Bockstein $\tilde \delta$ (see (\ref{eq:derivation-property})) together with its functoriality, we find
$$
\tilde \delta(p^\ast \alpha\cup q^\ast \beta)=-p^\ast \alpha\cup q^\ast \tilde \delta ( \beta) .
$$
By \cite[Corollary 7.9]{Sch-torsion1},
 $\tilde \delta(\beta)$ extends to  $Y$ and so
 $ 
 p^\ast \alpha\cup q^\ast \beta\in \widetilde G^{i+1} H^{2i} (F_{i-1}(X\times Y),\mu_{\ell^r}^{\otimes i+1}).
 $ 
This shows that $\Lambda$ is well-defined, which concludes the proof of the lemma.      
\end{proof}


\begin{lemma} \label{lem:deltaH^1}
Let $X$ be a smooth variety over an algebraically closed field $k$ and let $\ell$ be a prime invertible in $k$.
Then there is a canonical isomorphism
$$
\widetilde \delta(H^1(X,\mu_{\ell^r}))\cong A^1(X)[\ell^r]/\ell^rA^1(X)[\ell^{2r}].
$$
\end{lemma}
\begin{proof}
Since algebraic and homological equivalence coincides for divisors on smooth varieties,  
$$
A^1(X)[\ell^r]\cong H^2(X,\Z_\ell(1))[\ell^r]\cong \delta (H^1(X,\mu_{\ell^r})),
$$
where we use that any class in $H^2(X,\Z_\ell(1))[\ell^r]\cong \delta (H^1(X,\mu_{\ell^r}))$ is algebraic as it comes from a $\mu_{\ell^r}$-torsor and hence from a line bundle.
The lemma follows therefore from the compatibility of $\delta$ and $\tilde \delta$.
\end{proof}

Let $X$ and $Y$ be smooth varieties over an algebraically closed field $k$ and let $\ell$ be a prime invertible in $k$.
The natural exterior product map on cycles yields a map
\begin{align}  \label{def:exterior-product}
\frac{A^1(X)[\ell^r]}{\ell^rA^1(X)[\ell^{2r}] } \otimes A^i(Y)/\ell^r\longrightarrow  \frac{A^{i+1}(X\times Y)[\ell^r]}{\ell^r A^{i+1}(X\times Y)[\ell^{2r}]} ,\ \ [z_1]\otimes[z_2]\mapsto [z_1\times z_2] .
\end{align} 
To compare that map with $\Lambda$ from Lemma \ref{lem:Lambda}, we consider the diagram 
\begin{align} \label{diag:products}
\xymatrix{
\frac{A^1(X)[\ell^r]}{\ell^rA^1(X)[\ell^{2r}] }\otimes A^i(Y)/\ell^r\ar[r]^-{(\ref{def:exterior-product})}&  \frac{A^{i+1}(X\times Y)[\ell^r]}{\ell^r A^{i+1}(X\times Y)[\ell^{2r}]} \\
\tilde \delta( H^1(X,\mu_{\ell^r}))  \otimes \frac{ H^{2i-1}_{i-2,nr}( Y, \mu_{\ell^r}^{\otimes i})}{H^{2i-1}(Y,\mu_{\ell^r}^{\otimes i}) } \ar@{^{(}->}[u] \ar[r]^-{\Lambda}& \ar@{^{(}->}[u]  \frac{ H^{2i}(F_{i-1}(X\times Y), \mu_{\ell^r}^{\otimes i+1})}{\widetilde G^{i+1}H^{2i} (F_{i-1}(X\times Y), \mu_{\ell^r}^{\otimes i+1})},
}
\end{align}
where the vertical injection on the left comes from Lemma \ref{lem:deltaH^1} and Theorem \ref{thm:refined},  while the vertical injection on the right stems from Theorem \ref{thm:refined:Ti[ell]}.

\begin{lemma} \label{lem:diag-commutes}
The diagram (\ref{diag:products}) commutes. 
\end{lemma}
\begin{proof}
Let $\alpha\in H^1(X,\mu_{\ell^r})$ and $[\beta]\in   H^{2i-1}_{i-2,nr}( Y, \mu_{\ell^r}^{\otimes i})/H^{2i-1}(Y,\mu_{\ell^r}^{\otimes i})  $ with representative $\beta\in H^{2i-1} ( F_{i-1}Y, \mu_{\ell^r}^{\otimes i})$.
We claim first that 
\begin{align}\label{eq:delta(beta)}
\delta(\beta)\in \im(H^{2i} (  Y, \Z_{\ell}(i))\to H^{2i} ( F_{i-1}Y, \Z_\ell(i))) .
\end{align}
By (\ref{eq:lem:Fj}), it suffices to show that 
$$
\del (\delta(\beta))\in \bigoplus_{y\in Y^{(i)}}H^1(y,\Z_\ell(0))
$$
vanishes, which holds because the right hand side is torsion-free, while $  \delta(\beta) $ is torsion.

Mapping $\alpha\otimes [\beta]$ to the right upper corner  in (\ref{diag:products})  via
(\ref{def:exterior-product}), we get the element
$$
z_1\times [\del(\beta)] \in  A^{i+1}(X\times Y)[\ell^r]/\ell^rA^{i+1}(X\times Y)[\ell^{2r}]  ,
$$
where $z_1\in A^1(X)[\ell^r]$ is the unique class with $\cl_X^1(z_1)=\delta (\alpha)\in H^2(X,\Z_\ell(1))$ and where 
$$
\del(\beta)\in \bigoplus_{y\in Y^{(i)}}[y]\Z_\ell .
$$
On the other hand,  the image of $\alpha\otimes [\beta]$ in $ \frac{A^{i+1}(X\times Y)[\ell^r]}{\ell^r A^{i+1}(X\times Y)[\ell^{2r}]}$ via $\Lambda$ is represented by 
$$
\del \gamma\in \bigoplus_{z\in (X\times Y)^{(i+1)}}[z]\Z/\ell^r  ,
$$
where $\gamma\in H^{2i+1}(F_i(X\times Y),\mu_{\ell^r}^{\otimes i+1})) $ is a lift of $\tilde \delta(p^\ast \alpha\cup q^\ast \beta)$.
The derivation property of $\tilde \delta$ (see (\ref{eq:derivation-property})) and functoriality of the Bockstein $\tilde \delta$ yields
$$
\tilde \delta(p^\ast \alpha\cup q^\ast \beta)=p^\ast (\tilde \delta(\alpha))\cup q^\ast\beta- p^\ast  \alpha\cup q^\ast (  \tilde \delta (\beta) ) \in   H^{2i+1}(F_{i-1}(X\times Y),\mu_{\ell^r}^{\otimes i+1}) .
$$
By (\ref{eq:delta(beta)}),  $\delta(\beta)$ and hence the reduction $\tilde \delta (\beta)$ extends to a class on $Y$.  
Since $\alpha $ is a global class on   $X$,  we find that for any lift $\gamma' \in H^{2i+1}(F_i(X\times Y),\mu_{\ell^r}^{\otimes i+1}) $  of $p^\ast (\tilde \delta(\alpha))\cup q^\ast\beta$,
$$
[\del \gamma]=[\del \gamma']\in A^{i+1}(X\times Y)/\ell^r.
$$
Since $\tilde \delta(\alpha)$ is an algebraic class, it is supported on some divisor $D\subset X$.
Since $\beta$ is defined away from a codimension $i$ subset $Z\subset Y$,  $p^\ast (\tilde \delta(\alpha))\cup q^\ast\beta$ is defined away from the codimension $i+1$ subset $D\times Z$ and so we may pick a lift $\gamma'$ as above in such a way that
$$
\del \gamma'=\del p^\ast (\tilde \delta(\alpha))\cup q^\ast\beta=z_1\times \del \beta \in \bigoplus_{z\in (X\times Y)^{(i+1)}}[z]\Z/\ell^r  .
$$
This concludes the proof of the lemma.
\end{proof}

\section{Smooth specialization of refined unramified cohomology}  \label{sec:sp}

Let $\kappa$ be a  field and let $R=\mathcal O_{B,0}$ be the local ring of a smooth pointed curve $(B,0)$ over $\kappa$, where $0\in B$ is a $\kappa$-rational point.
Let $K=\Frac R$ and let $\ell$ be a prime invertible in $\kappa$.
The main result of this section is as follows.

\begin{proposition} \label{prop:specialization}
Let $R$ be as above, fix a uniformizer $\pi\in R$, and let $\mathcal X\to \Spec R$ be a smooth morphism with equi-dimensional fibres $X_0=\mathcal X\times_R\kappa$ and $X_\eta=\mathcal X\times_RK$.
There are linear specialization maps: 
\begin{align}\label{eq:sp-unramified} 
sp:H^i(F_jX_\eta,\mu_{\ell^r}^{\otimes n}) \longrightarrow H ^i(F_jX_0,\mu_{\ell^r}^{\otimes n})
\end{align}
with the following properties:
\begin{enumerate}
\setcounter{enumi}{-1}
\item $sp$ does not depend on $\pi$ if $\kappa$ is algebraically closed;\label{item:sp-0}
\item $sp$ respects  the filtration $F^\ast$;\label{item:sp-1}
\item if $\mathcal U\subset \mathcal X$ is an open subset with   $F_{j+1}X_\eta\subset U_\eta$ and $F_{j+1}X_0\subset U_0$, then any class $\alpha\in H^i(\mathcal U,\mu_{\ell^r}^{\otimes n})$ with restriction $[\alpha|_{U_\eta}]\in H_{j,nr}^i(X_\eta,\mu_{\ell^r}^{\otimes n})$ satisfies $sp([\alpha|_{U_\eta}])=[\alpha|_{U_0}] $;\label{item:sp-2}
\item the pushforwards from Lemma \ref{lem:pushforward} commute with specialization:  if $\mathcal Y\to \Spec R$ is a smooth morphism with equi-dimensional fibres $Y_\eta$ and $Y_0$,  and   $f:\mathcal X\to \mathcal Y$ is a proper $R$-morphism of pure relative dimension,  \label{item:sp-2.5}
then $(f|_{X_0})_\ast \circ sp=sp\circ (f|_{X_\eta})_\ast$;
\item $sp$   induces a specialization map between geometric fibres $X_{\overline \eta}:=X_\eta \times_K\overline K$ and $X_{\overline 0}:=X_0\times_\kappa \overline \kappa$: \label{item:sp-3}
$$
\bar{sp}:H ^i(F_jX_{\bar \eta},\mu_{\ell^r}^{\otimes n}) \longrightarrow H ^i(F_jX_{\overline 0},\mu_{\ell^r}^{\otimes n}).
$$
The map $\bar{sp}$ does not depend on the choice of $\pi$ and respects the filtration $F^\ast$. 
\end{enumerate}
\end{proposition}

\begin{corollary} \label{cor:specialization}
In the above notation, there are well-defined specialization maps
$$
sp:H^i_{j,nr}( X_\eta,\mu_{\ell^r}^{\otimes n}) \longrightarrow H ^i_{j,nr}( X_0,\mu_{\ell^r}^{\otimes n})
\ \ \ \text{and}\ \ \ 
\bar {sp}:H^i_{j,nr}( X_{\overline \eta},\mu_{\ell^r}^{\otimes n}) \longrightarrow H ^i_{j,nr}( X_{\overline 0},\mu_{\ell^r}^{\otimes n}) .
$$
\end{corollary}

\begin{remark}
In the above corollary, $sp$ may depend on the choice of the uniformizer $\pi\in R$, while $\bar {sp}$ is independent of that choice.
\end{remark}

\subsection{Construction - part 1} \label{subsec:specialization-smooth-family}

Let $\mathcal U\to \Spec R$ be a smooth  morphism with equi-dimensional fibres $U_0=\mathcal U\times_R \kappa$ and $U_\eta=\mathcal U\times_R K$.
By Remark \ref{rem:mathcal-X},  (\ref{eq:Gysin}) yields an exact sequence
\begin{align}\label{eq:Gysin-2}
H^i(\mathcal U,\mu_{\ell^r}^{\otimes n})\longrightarrow H^i(U_\eta,\mu_{\ell^r}^{\otimes n})\stackrel{\del}\longrightarrow H^{i-1}(U_0,\mu_{\ell^r}^{\otimes n-1})\longrightarrow H^{i+1}(\mathcal U,\mu_{\ell^r}^{\otimes n}) .
\end{align}

We fix a uniformizer $\pi\in R$. 
Then $\pi$ gives rise to a class in $ H^1(K,\mu_{\ell^r})\cong K^\ast/(K^{\ast})^{\ell ^r}$ and so we get a class $(\pi)\in H^1(U_\eta, \mu_{\ell^r})$ via pullback.
We then define 
$$
sp:H^i(U_\eta,\mu_{\ell^r}^{\otimes n})\longrightarrow H^{i}(U_0,\mu_{\ell^r}^{\otimes n}),\ \ \alpha\mapsto -\del((\pi)\cup \alpha) ,
$$
where $\del$ is the residue map in  (\ref{eq:Gysin-2}). 
(If the residue field $\kappa$ is not algebraically closed, then the class $(\pi)$ and hence the map $sp$ may depend on the choice of $\pi$, cf.\ proof of Proposition \ref{prop:specialization} below.)  

\begin{lemma}\label{lem:sp(a)=tildea}
If $\alpha\in H^i(U_\eta,\mu_{\ell^r}^{\otimes n})$ extends to a class $\tilde \alpha\in H^i(\mathcal U,\mu_{\ell^r}^{\otimes n})$, then
$ 
sp(\alpha)=\tilde \alpha|_{U_0} 
$.  
\end{lemma}
\begin{proof}
By (\ref{eq:Gysin-cup}),  
$ 
-\del((\pi)\cup \alpha)=-(\del(\pi) ) \cup \widetilde \alpha|_{U_0}
$ 
and so the result follows from the fact that $\del(\pi)=-1\in H^0(U_0,\mu_{\ell ^r}^{\otimes 0})=\Z/\ell^r$,  see e.g.\ \cite[(P6) in Proposition 6.6]{Sch-torsion1}. 
\end{proof}

\begin{lemma}\label{lem:sp(a)=0-Hochschild-Serre}
In the above notation
$$
sp(\ker( H^i(U_\eta,\mu_{\ell^r}^{\otimes n})\to H^i(U_{\overline \eta},\mu_{\ell^r}^{\otimes n}) ) )\subset \ker( H^{i}(U_0,\mu_{\ell^r}^{\otimes n}) \to  H^{i}(U_{\overline 0},\mu_{\ell^r}^{\otimes n})  ) .
$$ 
\end{lemma}
\begin{proof}
Replacing $\kappa$ by the algebraic closure $\overline \kappa$, we may assume that $\kappa$ is algebraically closed.
Then $K$ has cohomological dimension one and so the Hochschild--Serre spectral sequence \cite[p.\ 105, III.2.20]{milne} yields an exact sequence 
$$
0\longrightarrow H^1(K,H^{i-1}(U_{\overline \eta},\mu_{\ell^r}^{\otimes n}))\longrightarrow H^i(U_\eta,\mu_{\ell^r}^{\otimes n})\longrightarrow H^i(U_{\overline \eta},\mu_{\ell^r}^{\otimes n}) .
$$
For any glass $\beta\in H^j(U_\eta,\mu_{\ell^r}^{\otimes m})$, cup product with $\beta$ is compatible with the above spectral sequence.
The action on the $E_2$-page is induced by the action of $\beta$ on the coefficients and hence only depends on the image of $\beta$ in $H^j(U_{\bar \eta},\mu_{\ell^r}^{\otimes m})$. 
This description shows that cup product with the class $(\pi)\in H^1(K,\mu_{\ell^r})$ acts trivially on the  $E_2$-page of the above spectral sequence.
It thus follows from the above short exact sequence that for any class
 $\alpha\in  \ker( H^i(U_\eta,\mu_{\ell^r}^{\otimes n})\to H^i(U_{\overline \eta},\mu_{\ell^r}^{\otimes n}) )$, we have
$$
 (\pi)\cup \alpha=0\in  H^{i+1}(U_\eta,\mu_{\ell^r}^{\otimes n+1}) ,
$$ 
and so $sp(\alpha)=0$.
This proves the lemma.
\end{proof}

\subsection{Construction - part 2} \label{subsec:specialization-refined}


Let $R$ and $\pi$ be as above and let $\mathcal X\to \Spec R$ be a smooth  morphism with equi-dimensional fibres.
For any $j\geq i$, we define a specialization map
$$
sp:  H^i(F_jX_\eta,\mu_{\ell^r}^{\otimes n}) \longrightarrow  H^i(F_jX_0,\mu_{\ell^r}^{\otimes n})
$$
as follows.
Let $\alpha\in  H^i(F_jX_\eta,\mu_{\ell^r}^{\otimes n})$.
Then there is a closed subset $Z_\eta\subset X_{\eta}$ of codimension $>j$ such that $\alpha=[\alpha_{U_\eta}]$ is represented by a class
$$
\alpha_{U_\eta}\in H^i(U_\eta,\mu_{\ell^r}^{\otimes n} ),
$$
where $U_\eta=X_\eta\setminus Z_\eta$.
The closure $\mathcal Z\subset \mathcal X$ of $Z_\eta$ is automatically flat over $R$ and so the special fibre $Z_0$ has codimension $>j$ in $X_0$.
Let $\mathcal U:=\mathcal X\setminus \mathcal Z $ with special fibre $U_0=X_0\setminus Z_0$.
By Section \ref{subsec:specialization-smooth-family}, we get a class
$ 
\alpha_{U_0}:=sp(\alpha_{U_\eta})\in  H^i(U_0,\mu_{\ell^r}^{\otimes n} ) 
$ and define
$$
sp(\alpha)=[\alpha_{U_0}]\in   H^i(F_jX_0,\mu_{\ell^r}^{\otimes n}).
$$
Functoriality of the Gysin sequence with respect to open immersions (see Lemma \ref{lem:Gysin}) immediately shows that this definition is well-defined.

\begin{lemma} \label{lem:specialize-filtration-refined}
The specialization map
$$
sp:  H^i(F_jX_\eta,\mu_{\ell^r}^{\otimes n}) \longrightarrow  H^i(F_jX_0,\mu_{\ell^r}^{\otimes n}),\ \ \ \alpha\mapsto [\alpha_{U_0}]
$$
defined above is compatible with the filtration  $F^\ast$.
\end{lemma}
\begin{proof}
Let $m\geq j$
 and assume in the above notation that $\alpha \in F^m H^i(F_jX_\eta,\mu_{\ell^r}^{\otimes n} )$.
 This means that we may choose $\alpha_{U_\eta}\in H^i(U_\eta,\mu_{\ell^r}^{\otimes n} )$ in such a way that $Z_\eta$ has actually codimension $>m$ in $X_\eta$.
 But then the above construction immediately shows that $sp(\alpha)$ lifts to a class in $ H^i(F_mX_0,\mu_{\ell^r}^{\otimes n})$ and hence lies in $F^m H^i(F_jX_0,\mu_{\ell^r}^{\otimes n})$.
%
\end{proof}
 
\begin{lemma} \label{lem:sp-commutes-pushforwards}
Let $\mathcal Y\to \Spec R$ be another smooth morphism with equi-dimensional fibres $Y_\eta$ and $Y_0$ and let $f:\mathcal X\to \mathcal Y$ be a proper $R$-morphism of pure relative codimension $c:=\dim X_\eta-\dim Y_\eta$.
Then the following diagram commutes:
$$
\xymatrix{
 H^i(F_jX_\eta,\mu_{\ell^r}^{\otimes n})\ar[d]^{f_\ast} \ar[r]^{sp} &  H^i(F_jX_0,\mu_{\ell^r}^{\otimes n}) \ar[d]^{f_\ast}  \\
 H^{i+2c}(F_{j+c}Y_\eta,\mu_{\ell^r}^{\otimes n+c}) \ar[r]^{sp} &  H^{i+2c}(F_{j+c}Y_0,\mu_{\ell^r}^{\otimes n+c}) .
}
$$
\end{lemma}
\begin{proof}
This follows directly from the projection formula and the fact that the Gysin sequence is functorial with respect to proper pushforwards, see  Lemma \ref{lem:Gysin}.
\end{proof}

\subsection{Proof of Proposition \ref{prop:specialization}}

\begin{proof}[Proof of Proposition \ref{prop:specialization}] 
The existence of the specialization map together with item (\ref{item:sp-1}) follows from the construction in Section \ref{subsec:specialization-refined} and Lemma  \ref{lem:specialize-filtration-refined}.
Items (\ref{item:sp-2}) and (\ref{item:sp-2.5}) follow from  Lemmas   \ref{lem:sp(a)=tildea} and \ref{lem:sp-commutes-pushforwards}, respectively. 
By construction, $sp$ depends on the class $(\pi)\in H^1(K,\mu_{\ell^r})$  of $\pi$.
On the other hand, the map will not change if we replace $R$ by its completion, and so we may from now on assume that $R$ is complete.
If $\kappa$ is algebraically closed,  then $H^1(\Spec R,\mu_{\ell^r})\cong H^1(\kappa,\mu_{\ell^r})=0$ by \cite[Corollary VI.2.7]{milne} and so $\del:H^1(K,\mu_{\ell^r})\to H^0(\kappa,\Z/\ell^r)=\Z/\ell^r$ is an isomorphism, which implies that in this case the class $(\pi)\in H^1(K,\mu_{\ell^r})$  is independent of $\pi$.
Hence, $sp$ does not depend on the choice of $\pi$ if $\kappa$ is algebraically closed, as claimed in item (\ref{item:sp-0}).

To prove the existence of the specialization map on geometric fibres in item (\ref{item:sp-3}),  we may up to replacing $\kappa$ by its algebraic closure, assume that $\kappa$ is algebraically closed.
It follows from what we have said above that $sp$ is in this situation independent of the choice of $\pi$.
To prove (\ref{item:sp-3}), it thus suffices by Lemmas \ref{lem:sp(a)=0-Hochschild-Serre} and \ref{lem:specialize-filtration-refined} to show that  any class  $[\alpha]\in H ^i(F_jX_{\bar \eta},\mu_{\ell^r}^{\otimes n}) $ comes up to a finite base change from $H ^i(F_jX_{ \eta},\mu_{\ell^r}^{\otimes n}) $.
To prove this, let $U_{\overline \eta}\subset X_{\overline \eta}$ be an open subset with $F_{j}X_{\overline \eta}\subset U_{\overline \eta}$ and let $\alpha\in H^i(U_{\overline \eta},\mu_{\ell^r}^{\otimes n})$.
Taking the Galois closure of the complement of $U_{\overline \eta}\subset X_{\overline \eta} $, we may up to shrinking $U_{\overline \eta}$ assume that $U_{\overline \eta}=U_\eta\times \overline K$ for some open subset $U_\eta\subset X_\eta$ with $F_{j}X_\eta\subset U_\eta$.
Since $H^i(U_{\overline \eta},\mu_{\ell^r}^{\otimes n})$ is a finite group,  we may up to a finite base change also assume that the Galois group $G=\Gal(\overline K/K)$  acts trivially on $\alpha$.
Since $\kappa$ is algebraically closed,  $K$ has cohomological dimension $1$.
The Hochschild--Serre spectral sequence \cite[p.\ 105, III.2.20]{milne} thus shows that $\alpha$ lies in the image of
$
H^i(U_{ \eta},\mu_{\ell^r}^{\otimes n})\to H^i(U_{\overline \eta},\mu_{\ell^r}^{\otimes n})
$ 
as we want.
This concludes the proof of the proposition.
\end{proof}

\begin{proof}[Proof of Corollary \ref{cor:specialization}]
Since $H_{j,nr}^i(X,A(n))=F^{j+1}H^i(F_jX,A(n))$, the corollary follows  from
  Proposition \ref{prop:specialization}. 
\end{proof}

\begin{remark} \label{rem:sp}
The above specialization maps yield as a special case (where $j=0$)   maps 
\begin{align}\label{eq:sp-Galois}
sp:H^i(K(X_\eta),\mu_{\ell^r}^{\otimes n})\longrightarrow H^i(\kappa(X_0),\mu_{\ell^r}^{\otimes n})
\end{align}
that are well-known from Galois cohomology.  
On the other hand, the specialization maps between refined unramified cohomology from Corollary \ref{cor:specialization} seem 
 new even in the case $j=0$, where the groups in question coincide with  traditional unramified cohomology.
In fact, the situation is somewhat subtle: Unramified classes may in proper flat families specialize  via (\ref{eq:sp-Galois}) to ramified classes and this was the main technique to prove nontriviality of certain unramified classes in \cite[Section 6]{Sch-JAMS}.
The main point is that the families $\mathcal X\to \Spec R$ considered in \emph{loc.\ cit.\ }are flat but not smooth and our results here show that the ramification has to lie on the singular locus of $\mathcal X$ over $R$.  
\end{remark}

\section{A vanishing result} \label{sec:vanishing}

Let $\kappa$ be an algebraically closed field.
We assume for simplicity that $\kappa$ has characteristic zero. 
Let $R$ be the local ring of a smooth pointed curve $(B,0)$ over $\kappa$  with fraction field $K=\Frac R$. 
Let $\mathcal X$, $\mathcal Y$, and $\mathcal W$ be flat  $R$-schemes with equi-dimensional fibres and let 
$$
p:\mathcal W\longrightarrow \mathcal X\ \ \text{and}\ \ q:\mathcal W\longrightarrow \mathcal Y
$$
be $R$-morphisms with $c:=\codim_{\mathcal Y}(q(\mathcal W))$.
Assume that the following holds:
\begin{enumerate}  
\item $\mathcal X$ is regular, $\mathcal W$ is integral and normal,    $\mathcal Y\to \Spec R$ is smooth;\label{item:vanishing-2}
\item  $q$ is proper and generically finite onto its image $q(\mathcal W)\subset \mathcal Y$;\label{item:vanishing-3}
\item $p$ is dominant.\label{item:vanishing-4}
\end{enumerate}

Since $\kappa$ has characteristic zero, the generic fibre $W_\eta$ of $\mathcal W$ is generically smooth.
Lemma \ref{lem:funct-FjX} thus yields pushforward maps
$$
q_\ast:H^i(F_0W_\eta,\mu_{\ell^r}^{\otimes n})\longrightarrow  H^{i+2c}( F_cY_\eta,\mu_{\ell^r}^{\otimes n+c}) ,
$$
where $c:=\codim_{\mathcal Y}(q(\mathcal W))$.

\begin{lemma} \label{lem:vanishing}
In the above notation, let $\ell$ be a prime and let $\alpha\in H^i_{nr}(\mathcal X,\mu_{\ell ^r}^{\otimes n})$ be an unramified class on $\mathcal X$ whose restriction to the generic point of any component of the special fibre $X_0$ vanishes.
Then for any $\xi \in H^j(K(W_\eta),\mu_{\ell^r}^{\otimes m})$, the class
$$
 q_\ast(p^\ast\alpha\cup \xi) \in H^{i+j+2c}( F_cY_\eta,\mu_{\ell^r}^{\otimes n+m+c})
$$
lies in the kernel of  
$
sp: H^{i+j+2c}( F_c Y_\eta,\mu_{\ell^r}^{\otimes n+m+c})\to  H^{i+j+2c}( F_cY_0,\mu_{\ell^r}^{\otimes n+m+c})
$
from Section \ref{subsec:specialization-refined}.
\end{lemma}
\begin{proof}
Since $\kappa$ has characteristic zero and $\mathcal W$ extends to a normal $\kappa$-variety over some neighbourhood of $0\in B$, we may up to shrinking $\mathcal W$ assume that $\mathcal W\to \Spec R$ is smooth. 
By linearity of $q_\ast$, we may up to shrinking $\mathcal W$ assume that $W_0$ is irreducible.
By Lemma \ref{lem:sp-commutes-pushforwards},  it then suffices to show that 
$$
sp(p^\ast\alpha\cup \xi)=0\in H^{i+j}( F_0W_0,\mu_{\ell^r}^{\otimes n+m}).
$$
This vanishing in turn follows from the fact that the residue map
$$
\del:H^{i+j+1}(F_0W_\eta,\mu_{\ell^r}^{\otimes n+m+1})\longrightarrow H^{i+j}(\kappa(W_0),\mu_{\ell^r}^{\otimes n+m})
$$
factorizes through the cohomology of the completion $\widehat{\mathcal O_{\mathcal W,W_0}}$ of the local ring of $\mathcal W$ at the generic point of $W_0$ together with the claim that
$$
p^\ast\alpha=0\in H^i(\widehat{\mathcal O_{\mathcal W,W_0}},\mu_{\ell^r}^{\otimes n}).
$$
To prove this last claim, note that
  the restriction map
$$
H^i(\widehat{\mathcal O_{\mathcal W,W_0}},\mu_{\ell^r}^{\otimes n})\longrightarrow H^i(\kappa(W_0),\mu_{\ell^r}^{\otimes n})
$$
is injective (see  \cite[Corollary  VI.2.7]{milne}) and $p^\ast \alpha$ lies in the kernel of the above map because $\alpha$ vanishes on any component of the special fibre $\mathcal X$ by assumption.
(This last step uses that $p$ is dominant by assumption (\ref{item:vanishing-4}).)
This concludes the proof of the lemma. 
\end{proof}

\section{An injectivity theorem}


\begin{theorem} \label{thm:product-2}
Let $\kappa$ be an algebraically closed field of characteristic zero and let $R$ be the local ring of a smooth pointed curve $(B,0)$ over $\kappa$. 
Let $k$ be an algebraic closure of $\Frac R$ and let $\ell$ be a prime. 
Assume that there is a proper strictly semi-stable $R$-scheme $\mathcal X\to \Spec R$ with connected fibres of relative dimension two, such that the following holds, where $X:=\mathcal X\times_Rk$ denotes the geometric generic fibre: 
\begin{enumerate}[label=\textbf{C\arabic*}]
\item the restriction map $H^2_{nr}(\mathcal X,\mu_{\ell ^r} )\to H^2_{nr}(X ,\mu_{\ell ^r} )$ is surjective;\label{item:Br(X-eta)-to-Br(X)}
\item \label{item:trivial-restr-to-X0} for each component $X_{0i}$ of $X_0$, the restriction map
$
H^2_{nr}(\mathcal X,\mu_{\ell ^r} )\to H^2_{nr}(X_{0i},\mu_{\ell ^r} )
$ is zero.  
\end{enumerate} 

Then for any smooth projective variety $Y_{\kappa}$ over $\kappa$ with base change $Y=Y_\kappa\times_{\kappa} k$,  and for any free $\Z/\ell^r$-module $M\subset  \tilde \delta \left(H^1(X,\mu_{\ell^r}) \right) $,  the following composition is injective:
$$
M \otimes \frac{   H^{2i-1}_{i-2,nr}(Y_\kappa, \mu_{\ell^r}^{\otimes i})}{H^{2i-1}(Y_\kappa,\mu_{\ell^r}^{\otimes i}) } \longrightarrow M \otimes \frac{   H^{2i-1}_{i-2,nr}(Y , \mu_{\ell^r}^{\otimes i})}{H^{2i-1}(Y ,\mu_{\ell^r}^{\otimes i}) } \stackrel{\Lambda}\longrightarrow  \frac{  H^{2i}(F_{i-1}(X\times Y),\mu_{\ell^r}^{\otimes i+1})}{\widetilde G^{i+1}H^{2i} (F_{i-1}(X\times Y),\mu_{\ell^r}^{\otimes i+1})} ,
$$
where $\Lambda$ is the map  from Lemma \ref{lem:Lambda}.
\end{theorem} 


\begin{proof}  
Let   $\alpha_1,\dots ,\alpha_n\in H^1(X,\mu_{\ell^r} )$ such that the free $\Z/\ell^r$-module $M$ is given by
\begin{align}\label{eq:M}
M=  \bigoplus_{j=1}^n  \tilde \delta (\alpha_j)\Z/\ell^{r}\Z\subset  \tilde \delta \left(H^1(X,\mu_{\ell^r}) \right)   .
\end{align}
Let further 
\begin{align}\label{eq:beta_kappa}
\beta_{\kappa,j}\in H^{2i-1}(F_{i-1}Y_\kappa,\mu_{\ell^r}^{\otimes i})\ \ \ \ \text{for}\ \ \ j=1,\dots ,n
\end{align}
and denote the image of $\beta_{\kappa,j}$ in $H^{2i-1}(F_{i-1}Y,\mu_{\ell^r}^{\otimes i}) $ by $\beta_j$.
For a contradiction, we assume that
\begin{align}\label{eq:thm:inj:nonzero}
\sum_{j=1}^n \tilde \delta(\alpha_j)\otimes [\beta_{\kappa,j}]\neq 0\in M \otimes \frac{ H^{2i-1}_{i-2,nr}(Y_\kappa, \mu_{\ell^r}^{\otimes i})}{H^{2i-1}(Y_\kappa,\mu_{\ell^r}^{\otimes i})} ,
\end{align}
and
$$
\Lambda(\sum_{j=1}^n \tilde \delta(\alpha_j)\otimes [\beta_{j}] )=\sum_{j=1}^n [p^\ast \alpha_j\cup \beta_j]=0\in \frac{ H^{2i}(F_{i-1}(X\times Y),\mu_{\ell^r}^{\otimes i+1})}{\widetilde G^{i+1}H^{2i} (F_{i-1}(X\times Y),\mu_{\ell^r}^{\otimes i+1})} .
$$
By the derivation property for $\tilde \delta$  (see (\ref{eq:derivation-property})), the latter is equivalent to saying that   
\begin{align}\label{eq:thm:inj:extends}
\sum_{j=1}^n p^\ast  (\tilde \delta \alpha_j)\cup q^\ast \beta_j -\sum_{j=1}^n p^\ast(\alpha_j)\cup q^\ast(\tilde \delta \beta_j)\in F^{i+1} H^{2i+1}(F_{i-1}(X\times Y),\mu_{\ell^r}^{\otimes i+1}),
\end{align}
which by (\ref{eq:lem:Fj}) means that the above
class in $H^{2i+1}(F_{i-1}(X\times Y),\mu_{\ell^r}^{\otimes i+1})$ extends to a class on $X\times Y$.
The theorem will be proven if we derive a contradiction from (\ref{eq:thm:inj:nonzero}) and (\ref{eq:thm:inj:extends}).

Since $X$ is a smooth proper connected surface over an algebraically closed field, 
$$
  H^4(X,\mu_{\ell^r}^{\otimes 2})= \Z/\ell^r\cdot \cl_X^2(pt)
$$
where $\cl_X^2(pt)$ denotes the cycle class of a closed point on $X$.

\bigskip

\noindent
\textbf{Step 1.} For any $j_0\in \{1,\dots ,n\}$, there is a class $\overline{\tilde \delta(\alpha_{j_0})}\in H^2(X,\mu_{\ell^r})$ with
$$
\overline{\tilde \delta(\alpha_{j_0})}\cup 
\tilde \delta(\alpha_{j})=
\begin{cases}
\cl_X^2(pt)\ \ \ \ &\text{if $j=j_0$};\\
0\ \ \ \ \ \ \ &\text{otherwise}.
\end{cases}
$$
Moreover, the above property does not change if we add to $\overline{\tilde \delta(\alpha_{j_0})}$  a class that lifts to $H^2(X,\Z_\ell(1))$. 
\begin{proof} 
Since $X$ is a smooth proper surface over the algebraically closed field $k$, Poincar\'e duality implies that the cup product pairing on $H^2(X,\mu_{\ell^r})$ is perfect, see e.g.\ \cite[Theorem 24.1]{milne-2}. 
This  implies the existence of the classes $\overline{\tilde \delta(\alpha_{j_0})}$, because $\tilde \delta(\alpha_j)$ are $\Z/\ell^r$-linearly independent for $j=1,\dots ,n$ by (\ref{eq:M}).

By the compatibility of $\delta$ and $\tilde \delta$, the class $ \tilde \delta(\alpha_{j})$ is the reduction modulo $\ell^r$ of a torsion class in $H^2(X,\Z_\ell(1))$.
Since any class in $H^2(X,\Z_\ell(1))$ has trivial cup product pairing with a torsion class in $H^2(X,\Z_\ell(1))$, we find that the properties in question do not change if we add to $\overline{\tilde \delta(\alpha_{j_0})}$  a class that lifts to $H^2(X,\Z_\ell(1))$.
This concludes step 1.
\end{proof}

By (\ref{eq:beta_kappa}), there is an open subset $U_\kappa \subset Y_\kappa$ with $\codim_{Y_\kappa}(Y_\kappa\setminus U_\kappa)>i-1$
and such that  for all $j=1,\dots ,n$,  
$$
\beta_{\kappa , j}\in H^{2i-1}(U_\kappa,\mu_{\ell^r}^{\otimes i})\ \ \text{and}\ \ \ \beta_j\in H^{2i-1}(U,\mu_{\ell^r}^{\otimes i})
$$ 
where  $U=U_\kappa \times k$, and where by slight abuse of notation we do not distinguish between the above classes on the open subsets $U_\kappa$ and $U$ and their restrictions to  $F_{i-1}Y_\kappa$ and $F_{i-1}Y$, respectively. 

Assumption (\ref{eq:thm:inj:extends}) together with Lemma \ref{lem:les}   implies that there is a class 
$$
\xi\in \bigoplus_{w\in (X\times U)^{(i)}} H^1( w ,\mu_{\ell^r})
$$
such that
\begin{align}\label{eq:thm:inj:extends-2}
\gamma:=\sum_{j=1}^n p^\ast  (\tilde \delta \alpha_j)\cup q^\ast \beta_j -\sum_{j=1}^n p^\ast(\alpha_j)\cup q^\ast(\tilde \delta \beta_j)+\iota_\ast \xi\in  H^{2i+1}(F_{i}(X\times U),\mu_{\ell^r}^{\otimes i+1})
\end{align} 
extends to a class on $X\times Y$.
(Note that the above class lies on $F_i(X\times U)$, not on $F_i(X\times Y)$.)

\bigskip

\noindent
\textbf{Step 2.} The map 
$$
q_\ast: H^{2i+3}(F_{i}(X\times U),\mu_{\ell^r}^{\otimes i+2})\longrightarrow   H^{2i-1}(F_{i-2}U,\mu_{\ell^r}^{\otimes i})
$$
from Lemma \ref{lem:funct-FjX} satisfies
  for any $j_0\in \{1,\dots ,n\}$:
\begin{align} \label{eq:q_ast-extends}
q_\ast\left(  p^\ast \left( \overline{\tilde \delta(\alpha_{j_0})}\right) \cup  \gamma\right) = \beta_{j_0}+ q_\ast \left( p^\ast \left( \overline{\tilde \delta(\alpha_{j_0})}\right) \cup \iota_\ast \xi \right) \in  H^{2i-1}(F_{i-2}U,\mu_{\ell^r}^{\otimes i}).
\end{align}
\begin{proof}
This is  a consequence of  (\ref{eq:thm:inj:extends-2}), the computation in Step 1 and the fact that 
$$
q_\ast \left( p^\ast \left( \overline{\tilde \delta(\alpha_{j_0})} \cup \alpha_j\right) \cup q^\ast(\tilde \delta \beta_j)\right) =0 ,
$$
which follows from the projection formula, 
because $ q_\ast(p^\ast(\overline{\tilde \delta(\alpha_{j_0})}\cup  \alpha_j))=0$ for degree reasons. 
This concludes step 2.
\end{proof}

Recall that $U=U_\kappa\times_\kappa k$ and consider the smooth $R$-scheme $\mathcal U:=U_\kappa\times_\kappa R$.
Since $\kappa$ is algebraically closed,  applying item (\ref{item:sp-3}) of Proposition \ref{prop:specialization} to this family, we get a specialization map
$$
\bar{sp}:H^{2i-1}(F_{i-2}U,\mu_{\ell^r}^{\otimes i})\longrightarrow H^{2i-1}(F_{i-2}U_\kappa,\mu_{\ell^r}^{\otimes i}) .
$$   
\smallskip

\noindent
\textbf{Step 3.} 
We have
$$
\bar{sp}\left(q_\ast \left( p^\ast \left( \overline{\tilde \delta(\alpha_{j_0})}\right) \cup \iota_\ast \xi \right) \right)=0\in H^{2i-1}(F_{i-2}U_\kappa,\mu_{\ell^r}^{\otimes i}).
$$

\begin{proof}
By linearity, we may assume that $\xi\in H^1(\kappa(w),\mu_{\ell^r})$ for some $w\in (X\times U)^{(i)}$.
Since $X$ is a surface, $q(w)\in U$ has codimension at least $i-2$.
If $q(w)$ has codimension greater than $i-2$, then 
$$
q_\ast \left( p^\ast \left( \overline{\tilde \delta(\alpha_{j_0})}\right) \cup \iota_\ast \xi \right)=0\in  H^{2i-1}(F_{i-2}U,\mu_{\ell^r}^{\otimes i})
$$
by Lemma \ref{lem:les}
and we are done.
Hence, we may assume $q(w)\in U^{(i-2)}$.

The pullback of $\overline{\tilde \delta(\alpha_{j_0})} $ via $p^\ast $ factorizes through  the restriction 
$$ 
\left( \overline{\tilde \delta(\alpha_{j_0})}\right)|_{p(w)} \in H^2(\kappa(p(w)),\mu_{\ell^r}).
$$
If $p(w)$ is not the generic point of $X$, then the latter group vanishes by dimension reasons and we are done.
Hence, we may assume that $p(w)$ is the generic point of $X$.

Recall that $R$ is the local ring of a smooth pointed curve $(B,0)$ over $\kappa$.
We wish to perform a base change corresponding to a ring map $R\to R'$, where $R'$ is the local ring of a smooth pointed curve $(B',0')$ over $\kappa$ and $B'\to B$ is a finite morphism that maps $0'$ to $0$.
When we perform such a base change, the model $\mathcal X$ becomes singular, but it follows from \cite[Proposition 2.2]{hartl} that $\mathcal X\times_RR'$ can be made into a strictly semi-stable $R'$-scheme $\mathcal X'\to \Spec R'$ by repeatedly blowing up all non-Cartier components of the special fibre. 
The exceptional divisors introduced in these blow-ups are ruled surfaces over the algebraically closed field $\kappa$ and so they have trivial second unramified cohomology.
For this reason, assumptions (\ref{item:Br(X-eta)-to-Br(X)}) and (\ref{item:trivial-restr-to-X0})   remain true after such a base change.
We may thus in what follows freely apply base changes as above.

Up to a base change as described above, we may assume that the point $w$ is defined over $K=\Frac(R)$.
We may then consider the normalization
  $\mathcal W$  of the closure of $w$ in $\mathcal X\times_R\mathcal U$. 
The projections of $\mathcal X\times_R\mathcal U$ to the two factors yield natural maps
$$
p|_{\mathcal W}:\mathcal W\longrightarrow \mathcal X\ \ \text{and}\ \ q|_{\mathcal W}:\mathcal W\longrightarrow \mathcal U.
$$
By the above reduction step,  $p|_{\mathcal W}$ is dominant and $q|_{\mathcal W}$ is generically finite onto its image.
It follows that (\ref{item:vanishing-2})--(\ref{item:vanishing-4}) from Section \ref{sec:vanishing} are satisfied.

Since $H^1(\kappa(w),\mu_{\ell^r})\cong \kappa(w)^\ast/(\kappa(w)^\ast)^{\ell^r}$, we may up to a finite base change as above assume that $\xi$ extends to a class on the generic point of $\mathcal W$.
By assumption (\ref{item:Br(X-eta)-to-Br(X)}),   $  \overline{\tilde \delta(\alpha_{j_0})}$ lifts to a class in $H^2_{nr}(\mathcal X,\mu_{\ell^r})$ whose restriction to each component of the special fibre of $\mathcal X\to \Spec R$ vanishes by item (\ref{item:trivial-restr-to-X0}).
The vanishing claimed in step 3  is therefore a consequence of Lemma \ref{lem:vanishing}.
\end{proof}

By assumption, $\beta_{j}$ is the image of $\beta_{\kappa,j}$ via the natural map
$$
H^{2i-1}(F_{i-2}U_\kappa,\mu_{\ell^r}^{\otimes i})\longrightarrow H^{2i-1}(F_{i-2}U,\mu_{\ell^r}^{\otimes i}) ,
$$
where we recall $U=U_\kappa\times_\kappa k$.
The pullback of $\beta_{j,\kappa}$ to the product $\mathcal U:=U_\kappa\times _\kappa \Spec R$ is thus  a class that restricts to $\beta_j$ on the geometric generic fibre and to $\beta_{j,\kappa}$ on the special fibre.
Item (\ref{item:sp-2}) in Proposition \ref{prop:specialization}  thus implies
$$
\bar{sp}(\beta_j)=\beta_{\kappa,j}\in H^{2i-1}(F_{i-2}U_\kappa,\mu_{\ell^r}^{\otimes i})
$$
for all $j$.
Step 3 together with (\ref{eq:q_ast-extends}) in Step 2 thus imply
\begin{align} \label{eq:sp-bar-thm-inj}
\bar{sp}\left(q_\ast\left(  p^\ast \left( \overline{\tilde \delta(\alpha_{j_0})}\right) \cup  \gamma\right) \right) = \bar{sp}(\beta_{j_0})= \beta_{\kappa,j_0}  \in  H^{2i-1}(F_{i-2}U_\kappa,\mu_{\ell^r}^{\otimes i}).
\end{align}

Since $H^2(X,\mu_{\ell^r})\to H^2_{nr}(X,\mu_{\ell^r})$ is surjective by (\ref{eq:lem:Fj}),  $ p^\ast \left( \overline{\tilde \delta(\alpha_{j_0})}\right) $ extends to $X\times Y$.
The same holds for $\gamma$ by assumption.
Hence,
$$
q_\ast\left(  p^\ast \left( \overline{\tilde \delta(\alpha_{j_0})}\right) \cup  \gamma\right)\in  H^{2i-1}(F_{i-2}U,\mu_{\ell^r}^{\otimes i}) 
$$
extends to a class on $Y$.
Note that $F_{i-2}Y=F_{i-2}U$ and $F_{i-2}Y_\kappa=F_{i-2} U_\kappa$.
Item (\ref{item:sp-2}) (resp.\ (\ref{item:sp-3})) in Proposition \ref{prop:specialization} thus implies that the left hand side of (\ref{eq:sp-bar-thm-inj}) extends to a class on $Y_\kappa$ and so
$$
 \beta_{\kappa,j_0}\in  H^{2i-1}(F_{i-2}U_\kappa,\mu_{\ell^r}^{\otimes i})=H^{2i-1}(F_{i-2}Y_\kappa,\mu_{\ell^r}^{\otimes i})
$$
extends to a class in $Y_\kappa$.
This holds for all $j_0=1,\dots ,n$ and so the class in (\ref{eq:thm:inj:nonzero}) vanishes, which contradicts our assumptions.   
This concludes the proof of the theorem. 
\end{proof}


\begin{remark}
The proof of Theorem \ref{thm:product-2} shows that conditions (\ref{item:Br(X-eta)-to-Br(X)}) and (\ref{item:trivial-restr-to-X0}) can be slightly weakened as follows.
The surjectivity in (\ref{item:Br(X-eta)-to-Br(X)}) is only needed to lift the classes $\overline{\tilde \delta(\alpha_{j})}$ from Step 1 to $\mathcal X$.
These classes may by Step 2 be modified by the image of integral classes and so
(\ref{item:Br(X-eta)-to-Br(X)})  may be weakened to only ask that  the composition
$$
H^2_{nr}(\mathcal X,\mu_{\ell ^r} )\longrightarrow H^2_{nr}(X ,\mu_{\ell ^r} )\longrightarrow H^2_{nr}(X ,\mu_{\ell ^r} )/H^2_{nr}(X ,\Z_{\ell}(1))
$$
is surjective, or even weaker, that the images of the classes $\overline{\tilde \delta(\alpha_{j})}$ in $H^2_{nr}(X ,\mu_{\ell ^r})/H^2_{nr}(X ,\Z_{\ell}(1)) $ are contained in the image of the above map.

Item (\ref{item:trivial-restr-to-X0}) is only used to ensure that the lifts of the classes  $\overline{\tilde \delta(\alpha_{j_0})}$ to $\mathcal X$  from Step 1  restrict trivially to the components of the special fibre of $\mathcal X$ and it would be enough to replace (\ref{item:trivial-restr-to-X0}) by this more precise condition.
\end{remark}
 
Recall $E_{\ell^r}^i(Y)\subset A^i(Y)/\ell^r $ from (\ref{def:E-ell-main-2}).
The exterior product map (\ref{def:exterior-product})
induces a map
$$
\times:A^1(X)[\ell^r]\otimes E_{\ell^r}^i(Y)\longrightarrow A^{i+1}(X\times Y)[\ell^r],\ \ [z_1]\otimes [z_2]\mapsto [z_1\times z_2].
$$  

\begin{corollary} \label{cor:product-3}
In the notation of Theorem \ref{thm:product-2}, the kernel of the natural composition
$$
 A^1(X)[\ell]\otimes E_{\ell }^i(Y_\kappa)\longrightarrow A^1(X)[\ell]\otimes E_{\ell }^i(Y)\stackrel{\times}\longrightarrow A^{i+1} (X\times Y)[\ell ] \longrightarrow A^{i+1} (X\times Y)/\ell 
$$ 
is given by $
\left( \ell \cdot A^1(X)[\ell^{2}]\right) \otimes E_{\ell }^i(Y_\kappa) .
$ 
\end{corollary}
\begin{proof}
It is clear that any class in $\left( \ell \cdot A^1(X)[\ell^{2 }]\right) \otimes E_{\ell }^i(Y_\kappa)$ maps to zero in $A^{i+1} (X\times Y)/\ell$.
The converse implication follows by the  commutative diagram in (\ref{diag:products}) (see Lemma \ref{lem:diag-commutes}) from Theorem \ref{thm:product-2}, applied to $M=\tilde \delta(H^1(X,\mu_{\ell}))$. 
\end{proof}

\section{Flower pot degenerations of Enriques surfaces} \label{sec:degeneration-pg=0}

Let $k$ be an algebraically closed field of characteristic zero and let $R=k[[t]]$.
Following previous work of Kulikov and Perrson, Morrison classified all semi-stable degenerations $\mathcal X\to \Spec R$ of Enriques surfaces over $R$, see \cite[Corollary 6.2]{morrison}.
In contrast to Kulikov's original claim in \cite{kulikov}, it is not true that up to birational equivalence we can always assume that $K_{\mathcal X}$ is $2$-torsion.
In fact, there are 3 additional types of degenerations (called (ib), (iib), and (iiib) in \cite[Corollary 6.2]{morrison}) that do not admit an \'etale $2:1$ cover by a Kulikov degeneration of K3 surfaces.
The simplest example of such an exceptional degeneration of Enriques surfaces is given by type (ib), called flower pot degenerations, cf.\ \cite[Proposition 3.3.1(3)]{persson}. 
The central fiber $X_0$ of a flower pot degeneration is given by a rational surface $S_0$ with disjoint smooth rational (-4) curves $C_1,\dots ,C_r$, such that along each $C_i$, there is glued a chain of Hirzebruch surfaces $F_4$ with a $\CP^2$ as an end-component which is glued to the last $F_4$ along a smooth conic.
On the level of K3 covers, a flower pot degeneration corresponds to the degeneration of a K3 surface $T$ with a fixed point free involution $\iota$ to a nodal K3 surface $T_0$ and an involution $\iota_0$ which fixes exactly the nodes of $T_0$.

We will need the following evident properties for a flower pot degeneration $\mathcal X\to \Spec R$:
\begin{itemize}
\item if we perform a base change $t\to t^m$ and resolve the resulting family by blowing up repeatedly all non-Cartier components of the special fibre, then the new semi-stable degeneration of Enriques surfaces is again a flower pot degeneration;
\item each component of the central fibre $X_0$ is rational and the dual complex of $X_0$ is a tree (i.e.\ a finite connected undirected acyclic graph).
\end{itemize}
Explicit examples of flower pot degenerations  of Enriques surfaces have been constructed by Horikawa in \cite[\S 10.2]{horikawa} and by Persson in \cite[Appendix 2]{persson}.
The dual complex of the central fibre is in both constructions given by a straight line, i.e.\ the components of the special fibre form a chain of rational surfaces.

\begin{theorem} \label{thm:Enriques-flower-pot}
Let $\kappa$ be an algebraically closed field of characteristic zero and let $R$ be the local ring at a closed point of a smooth curve over $\kappa$.
Let $\mathcal X\to \Spec R$ be a strictly semi-stable degeneration of Enriques surfaces such that the special fibre $X_0$ is a flower pot as in \cite[type (ib), Corollary 6.2]{morrison}.
Then up to a base change, the restriction map
$$
H^2_{nr}(\mathcal X,\mu_2)\longrightarrow H^2_{nr}(X_{\overline \eta},\mu_2)
$$
to the geometric generic fibre $ X_{\overline \eta}$ is surjective.
\end{theorem}
\begin{proof}
Up to a base change (followed by resolving the resulting family by blowing up all non-Cartier components of the special fibre), we may assume that the monodromy action on $H^2(X_{\overline \eta},\mu_2)$ is trivial.
The Hochschild--Serre spectral sequence $$
E_2^{p,q}=H^p(G,H^q(X_{\bar \eta},\mu_2))\Longrightarrow H^{p+q}(X_{ \eta},\mu_2)
$$ 
degenerates at $E_2$, because the absolute Galois group $G$ of $\Frac R$ has cohomological dimension one.
Hence we get an exact sequence
$$
H^2(X_{ \eta},\mu_2)\longrightarrow H^2(X_{\overline \eta},\mu_2)\stackrel{d_2}\longrightarrow H^2(G,H^1(X_{\overline \eta},\mu_2)).
$$
Since $G$ has cohomological dimension one,  $H^2(G,H^1(X_{\overline \eta},\mu_2))=0$ and so we find that $H^2(X_{ \eta},\mu_2)\to  H^2(X_{\overline \eta},\mu_2)$ is surjective. 
Hence there is a class $\alpha\in H^2(X_{ \eta},\mu_2)$ whose image in $H^2(X_{\overline \eta},\mu_2)$ generates  
$$
H^2_{nr}(X_{\overline \eta},\mu_2) \cong \coker(\cl^1_{X_{\overline \eta}}:\CH^1(X_{\overline \eta})\to H^2(X_{\overline \eta},\mu_2)) \cong \Z/2.
$$
To prove the theorem, it suffices to show that $\alpha$ is unramified on $\mathcal X$ (as the latter implies that $\alpha$ lifts to a class on $H^2(\mathcal X,\mu_2)$, cf.\ Lemma \ref{lem:Gysin} and Remark \ref{rem:mathcal-X}).
To prove the latter, let  $X_{0i}$ for $i\in I$ denote the components of $X_0$.
There is a natural residue map
$$
\del:H^2(X_{ \eta},\mu_2)\longrightarrow \bigoplus_{i\in I} H^1(k(X_{0i}),\Z/2) 
$$
and we let
$$
\gamma_i:=\del \alpha\in H^1(k(X_{0i}),\Z/2) .
$$

Up to a ramified $2:1$ base change (followed by resolving the resulting family as above),  by the commutative diagram in \cite[p.\ 148]{CTO},  we may  assume that $\gamma_i=0$ for all components that have not been introduced by resolving the singularities introduced by the $2:1$ base change.
Since the dual graph of $X_0$ is a tree, we conclude
 that 
 whenever $X_{0i}$ meets $X_{0j}$ with $i\neq j$, then either $\gamma_i$ or $\gamma_j$ is zero.
 This implies that there is subset $J\subset I$ such that
 $$
 X_0':=\bigcup_{j\in J} X_{0j}\subset X_0
 $$
is smooth and such that $\alpha$ extends to an unramified class on $X\setminus X'_0$.
Since $\alpha$ has degree $2$, it actually extends to an honest class 
$$
\alpha'\in H^2(\mathcal X\setminus X'_0,\mu_2) ,
$$
cf.\ (\ref{eq:lem:Fj}).
Since $X'_0$ is smooth, the residue of $\alpha'$ satisfies
$$
\del \alpha'\in H^1(X'_0,\Z/2).
$$
Since each component of $X_0$ is rational, the same holds for the components of $X'_0$ and so $H^1(X'_0,\mu_2)=0$.
Hence, $\del \alpha'=0$ and we find that $\alpha'$ extends to a class in $H^2(\mathcal X ,\mu_2)$.
This concludes the proof of the theorem.
\end{proof}

\begin{proof}[Proof of Theorem \ref{thm:Enriques-deg-intro}]
By \cite[\S 10.2]{horikawa} or \cite[Appendix 2]{persson}, there is a  strictly semistable degeneration $\mathcal X\to \Spec R$ of Enriques surfaces such that the special fibre $X_0$ is a flower pot as in \cite[type (ib), Corollary 6.2]{morrison} and where $R$ is the local ring at a closed point of a smooth curve over an algebraically closed field of characteristic zero.
Then $H^2_{nr}(\mathcal X,\mu_{\ell^r})\cong \Br(\mathcal X)[\ell^r]$ and $H^2_{nr}( X_{\overline \eta},\mu_{\ell^r})\cong \Br( X_{\overline \eta})[\ell^r]$ by \cite[Proposition 4.2.3]{CT} and Remark \ref{rem:mathcal-X}.
Theorem \ref{thm:Enriques-deg-intro} follows therefore from Theorem \ref{thm:Enriques-flower-pot}.
\end{proof}

We end this section with the following result which reveals some subtlety of the geometry of degenerations as in Theorem \ref{thm:Enriques-deg-intro}. 

\begin{proposition}\label{prop:Br-neq-0}
Let $\mathcal X\to \Spec R$ be as in Theorem \ref{thm:Enriques-deg-intro}.
Then $\Br(X_0)[2]\neq 0$ while $\Br(X_{0i})=0$ for each component $X_{0i}$ of $X_0$. 
\end{proposition}
\begin{proof}
Since each $X_{0i}$ is a smooth projective ruled surface over an algebraically closed field,   $\Br(X_{0i})=0$ is clear and it suffices to prove $\Br(X_0)[2]\neq 0$.
Replacing $R$ by its completion, we may assume that $R$ is complete and so there is a non-canonical isomorphism $R\cong \kappa [[t]]$ because the residue field $\kappa$ has characteristic zero.

By the proper base change theorem \cite[Corollary 2.7]{milne}, we  get canonical isomorphisms
\begin{align} \label{eq:proper-basechange}
H^2 (\mathcal X,\mu_{2^r})\cong H^2 ( X_0,\mu_{2^r})\ \ \text{and}\ \ H^2 (\mathcal X,\Z_2(1))\cong H^2 ( X_0,\Z_2(1)) ,
\end{align}
where the latter uses \cite[(0.2)]{jannsen}.
Let $X_\eta=\mathcal X\times_RK$ where $K=\Frac R\cong \kappa((t))$.
The Gysin sequence yields an exact sequence
\begin{align}\label{eq:prop:Br}
\oplus_i \Z_2[X_{0i}]\longrightarrow  H^2 (\mathcal X,\Z_2(1))\longrightarrow H^2 (X_\eta,\Z_2(1))
\end{align}
where $X_{0i}$ denote the components of the special fibre $X_0$ and the first arrow is given by $\sum a_iX_{0i}\mapsto \sum a_ic_1(\mathcal O_{\mathcal X}(X_{0i}))$.
Note that $G:=\Gal(\bar K/K)$ has cohomological dimension $1$ because $R\cong \kappa[[t]]$ with $\kappa$ algebraically closed. 
Moreover,  Kummer theory shows
$$
H^1(G,\mu_{2^r})\cong \kappa((t))^\ast/(\kappa((t))^\ast)^{2^r}.
$$
This is a finite group (isomorphic to $\Z/2^r$) 
and so $R^1\lim H^1(G,\mu_{2^r})=0$.
Hence, $H^2_{cont}(G, \Z_2(1) )\cong \lim H^2(G, \mu_{2^r} )=0$ by \cite[(1.6)]{jannsen} because $G$ has cohomological dimension $1$.
Moreover,  $H^1 (X_{\bar \eta},\Z_2(1))=0$ because $X_{\bar \eta}$ is an Enriques surface.
The Hochschild--Serre spectral sequence for continuous \'etale cohomology (see \cite[(0.3)]{jannsen}) thus yields an isomorphism
$$
 H^2 (X_\eta,\Z_2(1))\cong H^2 (X_{\bar \eta},\Z_2(1))^{G} .
$$
Since $H^1(G,\mathbb G_m)=H^2(G,\mathbb G_m)=0$, the Hochschild--Serre spectral sequence yields similarly an isomorphism $\Pic(X_\eta)\cong \Pic(X_{\bar \eta})^G$.
Since $X_{\bar \eta}$ is an Enriques surface, $H^2 (X_{\bar \eta},\Z_2(1))\cong \Pic(X_{\bar \eta})\otimes \Z_2$ is algebraic.
We thus conclude from the above isomorphism that
$$
 H^2 (X_\eta,\Z_2(1))\cong\Pic(X_\eta)\otimes \Z_2.
$$ 
Taking closures of divisors on $X_\eta$ thus shows via (\ref{eq:prop:Br}) that $H^2 (\mathcal X,\Z_2(1))$ is algebraic.
Since algebraic classes restrict to algebraic classes, the second isomorphism in (\ref{eq:proper-basechange}) shows similarly that
$ 
 H^2 (X_0,\Z_2(1)) 
$ is algebraic.
Hence, (\ref{eq:Brauer}) implies
\begin{align}\label{eq:prop:Br-2}
\Br(\mathcal X)[2^r]\cong H^2 (\mathcal X,\mu_{2^r}) / H^2 (\mathcal X,\Z_2(1)) \ \ \text{and}\ \ \Br(X_0)[2^r]\cong H^2 (X_0,\mu_{2^r})/H^2 (X_0,\Z_2(1)).
\end{align}
We thus conclude from (\ref{eq:proper-basechange}) that there are canonical isomorphisms 
$$
\Br(\mathcal X)[2^r]\cong \Br(X_0)[2^r]
$$
for all $r\geq 0$.
By item (\ref{item:Enriques:surjective}) in Theorem \ref{thm:Enriques-deg-intro}, $\Br(\mathcal X)[2]\neq 0$   and so $ \Br(X_0)[2]\neq 0$, which proves the proposition. 
\end{proof}

\begin{remark}\label{rem:prop:Brauer}
The argument of the above proof shows more generally: if $R$ in Theorem \ref{thm:Enriques-deg-intro} is complete, then the natural restriction map $\Br(\mathcal X)_{\tors}\to \Br(X_0)_{\tors}$ is an isomorphism.  
\end{remark}
 
\begin{remark}
One can use Proposition \ref{prop:Br-neq-0} to check that several degenerations of Enriques surfaces as in \cite[Corollary 6.2]{morrison} (e.g.\ those of type (iia)) do not satisfy the conclusion of Theorem \ref{thm:Enriques-deg-intro}.
\end{remark}

\section{Proof of Theorems \ref{thm:JC} and  \ref{thm:main}}

\begin{proof}[Proof of Theorem \ref{thm:main}]
Let $Y$ be a smooth complex projective variety and let $X$ be an Enriques surface that is very general with respect to $Y$.
We aim to show that the exterior product map
$$
A^i(Y)/2\longrightarrow  A^{i+1}(X\times Y)/2,\ \ \ \ [z]\mapsto [K_X\times z]
$$
is injective.
Using the Künneth decomposition for cohomology with $\Z/2$-coefficients together with the fact that $K_X$ has nontrivial cycle class in $H^2(X,\mu_2)$,  we see that $[K_X\times z]=0$ implies that $[z]$ has trivial cycle class in $H^{2i}(Y,\mu_2^{\otimes i})$.
The problem thus reduces to the statement that
$$
E_2^i(Y) \longrightarrow   A^{i+1}(X\times Y)/2,\ \ \ \ [z]\mapsto [K_X\times z]
$$
is injective, where 
$ 
E_{2}^i(Y)= \ker\left( \cl_Y^i:A^i(Y)/2\to H^{2i}(Y,\mu_2^{\otimes i}) \right) 
$ 
from (\ref{def:E-ell-main-2}).

By a straightforward specialization argument
and because the moduli space of complex Enriques surfaces is irreducible,   it suffices to prove the result for some smooth projective Enriques surfaces $X$ over $\C$.

Since Chow groups modulo algebraic equivalence are countable, there is a countable algebraically closed field $\kappa\subset \C$ such that $Y=Y_{\kappa}\times \C$ for some smooth projective variety $Y_{\kappa}$ over $\kappa$ and such that the natural map
$$
E_2^i(Y_{\kappa})\longrightarrow E_2^i(Y)
$$
is surjective, hence  an isomorphism because $\kappa$ and $\C$ are algebraically closed.

Let $R:=\kappa[[t]]$.
By Theorem \ref{thm:Enriques-deg-intro}, we may up to enlarging $\kappa$ assume that there is a regular flat proper scheme $\mathcal X\to \Spec R$ whose generic fibre is a smooth Enriques surface and such that:
\begin{itemize}
\item there is a class $\alpha \in H_{nr}^2(\mathcal X,\mu_2^{\otimes 2})$ whose pullback to the geometric generic fibre is the unique nonzero Brauer class of the Enriques surface $X_{\overline \eta}$;
\item  the restriction of $\alpha$ to each component of the special fibre is trivial.
\end{itemize} 

Let $k$ be an algebraic closure of the fraction field of $R$.
Since $\kappa$ is countable, we may assume that $k\subset \C$.

It follows from  Corollary \ref{cor:product-3} that 
$$
E_2^i(Y_{\kappa}) \longrightarrow  A^{i+1}(X_{\overline \eta}\times Y_k)/2 ,\ \ \ \ [z]\mapsto [K_X\times z]  
$$ 
is injective.
Since $k\subset \C$, we may consider the base change $X:=X_{\overline \eta}\times_k \C$.
Since $k\subset \C$ is an extension of algebraically closed fields, a well-known and straightforward specialization argument shows that 
$$
A^{i+1}(X_{\overline \eta}\times Y_k)/2\longrightarrow  A^{i+1}(X\times Y)/2 
$$
is injective.
Altogether, we thus see that
$$
E_2^i(Y)\cong E_2^i(Y_{\kappa}) \longrightarrow  A^{i+1}(X\times Y)/2  ,\ \ \ \ [z]\mapsto [K_X\times z]  
$$ 
is injective.
Here the Enriques surface $X$ is somewhat special, but as noted above, this also shows that the map in question is injective for a very general Enriques surface in place of $X$.
This concludes the proof.
\end{proof}

\begin{proof}[Proof of Corollary \ref{cor:main}]
Note that $K_X\in A^1(X)$ is $2$-torsion and the Jacobian variety of $X$ is trivial.
It follows that the Deligne cycle class (see e.g.\ \cite{Voi-unramified}) of $K_X$ is $2$-torsion.
Hence, 
 for $z\in A^i(Y)$ with $[z]\in
E_{2}^i(Y)= \ker\left( \cl_Y^i:A^i(Y)/2\to H^{2i}(Y,\mu_2^{\otimes i}) \right) 
$, the class $[K_X\times z]\in A^{i+1}(X\times Y)$ has trivial Deligne cycle class, because the Deligne cycle class of $K_X$ is $2$-torsion, while that of $z$ is divisible by $2$.
It follows that the exterior product map  
$$
E_2^i(Y) \longrightarrow  A^{i+1}(X\times Y) ,\ \ \ \ [z]\mapsto [K_X\times z]
$$
lands in the subspace of classes of the Griffiths group with trivial Abel--Jacobi invariant.
Since the image of the above map is clearly $2$-torsion, we find that its
  image is contained in
$$
\mathcal T^{i+1}(X\times Y)[2]\subset \Griff^{i+1}(X \times Y)\subset A^{i+1}(X\times Y) .
$$
Corollary \ref{cor:main} follows therefore from Theorem \ref{thm:main}.
\end{proof}

\begin{proof}[Proof of Theorem \ref{thm:JC}]
Let $C\subset \CP^2_\C$ be a very general quartic curve. 
Since the group of algebraically trivial cycles modulo rational equivalence is divisible, we have $\CH^2(Y)/2\cong A^2(Y)/2$.
Hence, \cite{totaro-chow} shows that $\Griff^2(JC)/2$ is infinite. 
Theorem \ref{thm:main} implies that for a very general Enriques surface $X$, the map
$$
\Griff^2(JC)/2\longrightarrow \Griff^3(X\times JC)[2],\ \ [z]\mapsto [K_X\times z]
$$
is injective.
Hence, $\Griff^3(X\times JC)$ has infinite $2$-torsion, as we want.
\end{proof}

\begin{remark}
Replacing \cite{totaro-chow} in the above proof by \cite{diaz-JAG},  we see that Theorem \ref{thm:JC} remains true if $JC$ is replaced by the product of three very general elliptic curves.
\end{remark}

\begin{corollary}\label{cor:torsion}
For any $n\geq 5$ and any $3\leq i\leq n-2$, there is a smooth complex projective $n$-fold $X$  with infinite torsion (in fact $2$-torsion) in $\Griff^i(X)$.
\end{corollary}
\begin{proof}
For $n\geq 5$, the projective bundle formula shows that the smooth complex projective variety $X\times JC\times \CP^{n-5}$, where $C$ and $X$ are as in the proof of Theorem \ref{thm:JC}, has infinite $2$-torsion in $\Griff^{i}$ for all $3\leq i\leq n-2$.
This proves the corollary.
\end{proof}

\begin{remark}
It remains open whether $\Griff^i(X)$ may have infinite torsion for $i=2$ or $n-1$; but recall that the $n$-torsion subgroup of $\Griff^2(X)$ is finite for any  $n\geq 1$, see  \cite[\S 18]{MS}.
\end{remark}

\section{Further applications}


In \cite{Voi-unramified,Ma}, Voisin and Ma showed that for any smooth complex projective variety $X$ whose Chow group of zero-cycles is supported on a threefold, there is a short exact sequence
$$
0\longrightarrow \left(  \frac{H^5(X,\Z)}{N^2H^5(X,\Z)} \right)_{\tors} \longrightarrow H^4_{nr}(X,\Q/\Z) \longrightarrow \mathcal T^3(X)\longrightarrow 0,
$$
where $N^\ast H^5(X,\Z)$ denotes the coniveau filtration and 
 where $\mathcal T^3(X)\subset \Griff^3(X)$ denotes the subgroup of torsion classes with trivial transcendental Abel--Jacobi invariants.
All three terms in the above sequence are birational invariants of $X$.
In \cite[\S 4]{Voi-unramified}, Voisin considered the question whether it can happen that the first term vanishes, while the unramified cohomology group in the middle is nonzero.
In \emph{loc.\ cit.,} Voisin showed that the generalized Hodge conjecture implies the existence of such examples; by the generalized Bloch conjecture, her examples should actually satisfy $\CH_0(X)=\Z$. 
By a result of Ottem and Suzuki \cite{OS},  
Theorem \ref{thm:main}  
answers Voisin's question unconditionally.


\begin{corollary}\label{cor:CH-0=0-2}
For any $n\geq 5$, there is a smooth complex projective $n$-fold $X$ with $\CH_0(X)\cong \Z$ such that 
$$
N^2H^5(X,\Z)= H^5(X,\Z)\ \ \text{and}\ \ H^4_{nr}(X,\Q/\Z)\cong  \mathcal T^3(X)= \Griff^3(X)_{\tors} \neq 0.
$$
\end{corollary}
\begin{proof} 
By \cite{OS}, there is a smooth complex projective threefold $Y$, given as a pencil of Enriques surfaces, such that $\CH_0(Y)\cong \Z$, $H^\ast(Y,\Z)$ is torsion-free and $Y$ admits a non-algebraic Hodge class $\alpha\in H^{2,2}(Y,\Z)$ such that $2\alpha$ is algebraic. 
It thus follows from Theorem \ref{thm:main} that  for any very general Enriques surface $X$, the product 
$$
Z:= X\times Y\times \CP^{n-5} 
$$ 
contains a nonzero $2$-torsion class in $\mathcal T^3(Z)$.
Moreover, $\CH_0(Z)\cong \Z$, because  $X$ has trivial Chow group of zero-cycles.
Since $Y$ has torsion-free cohomology, the K\"unneth formula applies and shows that $H^5(Z,\Z)$ decomposes as
$$
  ( H^3(X,\Z)\otimes H^2(Y\times \CP^{n-5},\Z)) \oplus ( H^2(X,\Z)\otimes H^3(Y\times \CP^{n-5},\Z)) \oplus H^5(Y\times \CP^{n-5},\Z) ,
$$
where we used $b_1(Y)=0$, since $\CH_0(Z)\cong \Z$, cf.\ \cite{BS}.
Since $\CH_0$ of $X$ and $Y$ are supported on a point, the  positive degree integral cohomology of $X$ and $Y$ is contained in $N^1$, cf.\ \cite[Proposition 3.3]{CTV}.
Using this, the above decomposition (together with the Künneth decomposition for $Y\times \CP^{n-5}$) easily shows that  
$ 
N^2H^5(Z,\Z)=  H^5(Z,\Z).
$  
Since $\CH_0(Z)\cong \Z$, this implies by \cite[Corollary 0.3]{Voi-unramified} that
$
H^4_{nr}(Z,\Q/\Z)\cong \mathcal T^3(Z). 
$ 
Moreover, $
N^2H^5(Z,\Z)=H^5(Z,\Z)
$ implies that the intermediate Jacobian $J^5(Z)$ of $Z$ is generated by the images of $J^1(\widetilde W)$, where $W\subset Z$ runs through all subvarieties of codimension $2$ and $\widetilde W$ denotes a resolution of $W$.
In particular, the transcendental intermediate Jacobian $J^5_{tr}(Z)$ vanishes, and so
$ 
\mathcal T^3(Z)  = \Griff^3(Z)_{\tors}.
$ 
This concludes the proof of the corollary.
\end{proof}

The condition $\CH_0(X)\cong \Z$ means that $X$ admits a rational decomposition of the diagonal.
By \cite[Theorem 1(ii)]{BS},
$$
\CH_0(X)\cong \Z\ \  \Longrightarrow \ \ \Griff^2(X)=0.
$$
This implication fails for $\Griff^3$, because we may blow-up varieties with $\CH_0=\Z$ along a smooth subvariety with nontrivial $\Griff^2$.
We are however not aware of any other construction that would yield varieties with small Chow groups of zero-cycles but nontrivial Griffiths groups. 
For instance, to the best of our knowledge, it was previously not known
 whether varieties with a rational decomposition of the diagonal  always admit a birational model with trivial Griffiths groups.
The following consequence of the above corollary solves that problem.

\begin{corollary}\label{cor:CH-0=0}
For any $n\geq 5$, there is a smooth complex projective $n$-fold $X$ 
such that any smooth complex projective variety $X'$ that is birational to $X$ satisfies
$$
\CH_0(X')\cong \Z\ \ \text{and}\ \ 
\Griff^3(X')\neq 0.
$$
\end{corollary}

\begin{proof}  
The result follows directly from Corollary \ref{cor:CH-0=0-2} and the fact that $\mathcal T^3(Z)$ is a birational invariant of smooth complex projective varieties, see \cite[Lemma 2.2]{Voi-unramified}.  
\end{proof}

\begin{corollary}\label{cor:Kummer}
Let $JC$ be the Jacobian of a smooth quartic $C\subset \CP^2_\C$ defined over $\Q$, with good reduction at $2$ and with associated Kummer variety $Y=\widetilde{JC}/\pm$.
Then for any very general Enriques surface $X$ over $\C$: 
$$
\Griff^3(X\times Y )[2]\neq 0 \ \ \text{and} \ \ \Griff^3(X\times Y)/2\neq 0 .
$$   
\end{corollary}
\begin{proof}[Proof of Corollary \ref{cor:Kummer}]
Let $C$ be a smooth plane curve of degree $4$ defined over $\Q$ and with good reduction at $2$.
Then $JC$ has good reduction at $2$ and \cite[Corollary 2.13]{diaz} implies that the Kummer variety $Y=\widetilde{JC/\pm}$ associated to $JC$ is a smooth complex projective variety with $H^3_{nr}(Y,\mu_2)\neq 0$.
As noted in \emph{loc.\ cit.\ }the Chow group of $Y$ is supported on a surface (see \cite[\S 4, Example (1)]{BS}) and the integral cohomology of $Y$ is torsion-free, so that \cite{CTV} implies that there is a non-algebraic non-torsion Hodge class $\alpha\in H^{2,2}(Y,\Z)$ such that $2\alpha$ is algebraic, cf.\ \cite[Corollary 3.3]{diaz}.
For any very general complex projective Enriques surface $X$, Theorem \ref{thm:main} thus implies that $\Griff^3(X\times Y)[2]\neq 0$  and $\Griff^3(X\times Y)/2\neq 0$, as we want. 
\end{proof}

\section*{Acknowledgements} 

I am grateful  to Theodosis Alexandrou,  Olivier Benoist,  Klaus Hulek, Andreas Rosenschon,  Matthias Sch\"utt, and Burt Totaro  for comments and discussions.
Thanks to the referee for his or her comments.
This project has received funding from the European Research Council (ERC) under the European Union's Horizon 2020 research and innovation programme under grant agreement No 948066.


\end{document}